\documentclass[12pt, a4paper]{amsart}
\usepackage{headers}

\title[Multiplicity One Theorems over Positive Characteristic]{Multiplicity One Theorems over Positive Characteristic}
\author{Dor Mezer}
\date{\today}


\begin{document}
\begin{abstract}
In \cite{AGRS} a multiplicity one theorem is proven for general linear groups, orthogonal groups and unitary groups ($GL, O,$ and $U$) over $p$-adic local fields. That is to say that when we have a pair of such groups $G_n\ss G_{n+1}$, any restriction of an irreducible smooth representation of $G_{n+1}$ to $G_n$ is multiplicity free.
This property is already known for $GL$ over a local field of positive characteristic, and in this paper we also give a proof for $O,U$, and $SO$ over local fields of positive odd characteristic. These theorems are shown in \cite{GGP} to imply the uniqueness of Bessel models, and in \cite{CS} to imply the uniqueness of Rankin-Selberg models. We also prove simultaniously the uniqeuness of Fourier-Jacobi models, following the outlines of the proof in\cite{Sun}.\\
By the Gelfand-Kazhdan criterion, the multiplicity one property for a pair $H\leq G$ follows from the statement that any distribution on $G$ invariant to conjugations by $H$ is also invariant to some anti-involution of $G$ preserving $H$.
This statement for $GL, O$, and $U$ over over $p$-adic local fields is proven in \cite{AGRS}.
An adaptation of the proof for $GL$ that works over of local fields of positive odd characteristic is given in \cite{GL case}.
In this paper we give similar adaptations of the proofs of the theorems on orthogonal and unitary groups, as well as similar theorems for special orthogonal groups and for symplectic groups. Our methods are a synergy of the methods used over characteristic 0 (\cite{AGRS}, \cite{Sun}, \cite{Wald}) and of those in \cite{GL case}.
\end{abstract}
\maketitle

\section{Introduction}
Let $\FF$ be a local field of characteristic different from 2.
Let $V$ be a vector space of dimension $n$ over $\FF$. Let $W:=V\oplus \FF v_{n+1}$ be an $n+1$-dimensional vector space containing it.
Assume we have a non-degenerate symmetric bilinear form on $W$, with respect to which $V$ is orthogonal to $v_{n+1}$. Let $O(V),O(W)$ be the orthogonal groups of $V,W$ respectively, and let $SO(V), SO(W)$ denote the respective special orthogonal groups.

The following two theorems are among the main theorems proved in this paper:
\begin{theorem} \label{mult 1 for O}
Let $\pi$ and $\rho$ be irreducible smooth representations of $O(W)$ and $O(V)$ respectively. Then
$$\dim\mathrm{Hom}_{O(V)}(\pi,\rho)\leq 1.$$
\end{theorem}

\begin{theorem} \label{mult 1 for SO}
Let $\pi$ and $\rho$ be irreducible smooth representations of $SO(W)$ and $SO(V)$ respectively. Then
$$\dim\mathrm{Hom}_{SO(V)}(\pi,\rho)\leq 1.$$
\end{theorem}

There is a unitary analog for the first of these theorems which we shall now formulate.
First, Let $\KK$ be an extension of $\FF$ of degree 2. Let $V$ be a vector space of dimension $n$ over $\KK$. Let $W:=V\oplus \KK v_{n+1}$ be an $n+1$-dimensional vector space containing it.
Assume we have a non-degenerate Hermitian form on $W$, with respect to which $V$ is orthogonal to $v_{n+1}$. Let $U(V),U(W)$ be the unitary groups of $V,W$ respectively. That is, the group of linear transformations preserving the Hermitian form.

The analogue of Theorem \ref{mult 1 for O} that we also prove in this paper is the following:
\begin{theorem} \label{mult 1 for U}
Let $\pi$ and $\rho$ be irreducible smooth representations of $U(W)$ and $U(V)$ respectively. Then
$$\dim\mathrm{Hom}_{U(V)}(\pi,\rho)\leq 1.$$
\end{theorem}

Let us use a uniform notation for all three theorems. Denote by $G(V)$ the group $O(V), SO(V)$ or $U(V)$, and by $G(W)$ the group $O(W), SO(W)$ or $U(W)$.
We have an action of $G(V)$ on $G(W)$ by conjugation.
Let us define an anti-involution $\sigma$ of $G(W)$ in all three cases:\\
In the case of $O$, define $\sigma: g\mapsto g^{-1}$.\\
In the case of $SO$, choose $T\in G(V)$ of order 2 with $\det T = (-1)^{\floor{\frac{n+1}{2}}}$ - One may choose such an element by taking a basis of $V$ with respect to which the symmetric form is diagonal, and in this basis take a diagonal matrix of $\pm 1$ with the appropriate pairity of $-1$ entries. Define $\sigma: g\mapsto Tg^{-1}T$.\\
In the case of $U$, choose a basis of $W$ for which the Hermitian product of all pairs lies in $\FF$ (for example by choosing a basis that diagonalizes the Hermitian form). Then we have an involution $T:v\mapsto \bar{v}$ (writing $v$ as a vector in this basis). Define an anti-involution of $G(W)$ by $\sigma: g\mapsto Tg^{-1}T$.

Let $G(V)$ act on $G(W)$ by conjugation. The following theorem implies Theorem \ref{mult 1 for J} using the Gelfand-Kazhdan criterion:
\begin{theorem} \label{G(V), G(W) dist}
Any $G(V)$-invariant distribution on $G(W)$ is also invariant under $\sigma$.
\end{theorem}
The implication is proven both in \cite[Appendix B]{Sun} and in section 1 of \cite{AGRS} in zero characteristic, and the same proofs apply verbatim in arbitrary odd characteristic.

We also prove another theorem which we shall now describe, given in \cite{Sun} for characteristic 0. One may also look there for more extensive explanations about the basic notations and definitions used.\\
Let $A$ be a finite dimensional commutative involutive algebra over $\FF$, and let $V$ be a finitely generated $A$-module. Let $\epsilon = \pm 1$ and let $\tau$ be the involution of $A$. Assume that $V$ is equipped with a non-degenerate $\epsilon$-Hermitian form, i.e. a non-degenerate $\FF$-bilinear map $<\cdot, \cdot>:V\times V\to A$ satisying $A$-linearity in the first argument, and $<v, u> = \epsilon <u, v>^\tau$.
Denote by $S$ the group of all $A$-module automorphisms of $V$ which preserve this form. It is a finite product of general linear groups, unitary groups, orthogonal groups, and symplectic groups.
Denote by $A^{\tau = -\epsilon}$ the subset of $A$ of elements $a$ satisfying $a^\tau = -\epsilon a$.
Let $H$ be the Heisenberg group defined as $\{(v, t)|v\in V, t\in A^{\tau = -\epsilon}\}$ with multiplication
$$(v,t)(v',t') = (v + v', t + t' + \frac{<v, v'>}{2} - \frac{<v', v>}{2}).$$
We have a natural action of $S$ on $H$. Denote by $J:=H\rtimes S$ the semidirect product of $H$ and $S$ with respect to this action. We prove in this paper the following theorem.
\begin{theorem} \label{mult 1 for J}
Let $\pi$ and $\rho$ be irreducible smooth representations of $J$ and $S$ respectively. Then
$$\dim\mathrm{Hom}_{S}(\pi,\rho)\leq 1.$$
\end{theorem}

Again we shall use the method of Gelfand and Kazhdan.
Choose an $\FF$-linear involution $\sigma$ of $V$ such that $<\sigma u, \sigma v> = <v, u>$ - To see that it is always possible, write $G$ as a product of unitary, orthogonal, symplectic and general linear groups with the corresponding orthogonal decomposition of $V$ over $\FF$. In each of these cases we are able to construct such an involution, and we can take the product of these involutions as $\sigma$.
The involution $\sigma$ extends to an anti-involution of $H$ by $(v, t)\mapsto (-\sigma v, t)$, which together with the anti-involution $g\mapsto \sigma g^{-1}\sigma$ of $S$ gives an anti-involution of $J$ (which we shall also call $\sigma$).
Let $S$ act on $J$ by conjugation. The following theorem implies Theorem \ref{mult 1 for J} using the Gelfand-Kazhdan criterion:
\begin{theorem}\label{J dist}
Any $S$-invariant distribution on $J$ is also invariant under $\sigma$.
\end{theorem}
The implication is proven in \cite[Appendix B]{Sun} in zero characteristic, and the same proof applies verbatim in arbitrary odd characteristic.\\
Sections \ref{reformulations}-\ref{single stratum} are dedicated to the simultaneous proof of Theorems \ref{G(V), G(W) dist} and \ref{J dist}.

\begin{remark}
One could ask whether an analogue of Theorem \ref{mult 1 for U} holds for the subgroup $SU$ of Hermitian operators with determinant $1$. However, the answer turns out to be negative even for the most simple case of $\dim V = 1$. In this case, $SU(V) = \{1\}$, while $SU(W)$ is non-commutative, showing that the two groups do not satisfy a multiplicity one property.
\end{remark}

\subsection{Corollaries of the main theorems}
There are corollaries of Theorems \ref{mult 1 for O}, \ref{mult 1 for SO}, \ref{mult 1 for U} and \ref{mult 1 for J} which were shown over characteristic 0, and the proofs of implication still apply over positive characteristic.
We list them here with references.

The following theorem appears in \cite{Sun} as Theorem B, and its implication from Theorem \ref{mult 1 for J} is shown in Appendix A of the same.
\begin{theorem}
Let $G$ denote one of the groups $GL(n), U(n)$, and $Sp(2n)$ (note that $U(n)$ is dependent on the choice of a Hermitian form), regarded as a subgroup of $Sp(2n)$ as usual. Let $\tG$ be the double cover of $G$ induced by the metaplectic cover $\widetilde{Sp}(2n)$ of $Sp(2n)$. Denote by $\omega_\psi$ the smooth oscillator representation of $\widetilde{Sp}(2n)$ corresponding to a non-trivial character $\psi$ of $\FF$. Then for any irreducible smooth representation $\pi$ of $G$, and any genuine irreducible smooth representation $\pi'$ of $\tG$, one has that
$$\dim \operatorname{Hom}_G(\pi'\otimes\omega_\psi \otimes \pi, \CC)\leq 1.$$
\end{theorem}

The following theorem appears in \cite{CS}, along with its implication from the above theorem (see \cite{CS} for the definitions and notation used).
\begin{theorem}[Uniqueness of Rankin-Selberg models]
For all irreducible smooth representations $\pi$ of $GL(n)$ and $\sigma$ of $GL(r)$, and for all generic characters $\chi$ of the $r$-th Rankin-Selberg subgroup $R_r$ of $GL(n)$, one has that
$$\dim\operatorname{Hom}_{R_r}(\pi\otimes\sigma, \chi)\leq 1.$$
\end{theorem}
This following two theorems appear in \cite[Chapters 12-16]{GGP}, along with their implications from \ref{mult 1 for O}, \ref{mult 1 for U}, \ref{mult 1 for J}, and the multiplicity 1 property for $GL(n+1), GL(n)$ which is proven for positive characteristic in in \cite[Theorem 1.2]{GL case} and in \cite{AAG}.
\begin{theorem}[Uniqueness of Bessel Models]

Let $V$ be a linear space with a symmetric or Hermitian (including the case $\KK=\FF\times \FF$) form. Denote the respective orthogonal or unitary group $G(V)$. Let $W$ be a subspace of odd codimension on which the form is non-degenerate and so that $W^\perp$ is split. Let $H$ be the Bessel group corresponding to $W$, considered as a subgroup of $G(V)\times G(W)$, and let $\nu$ be a generic character of $H$. Then for any irreducible smooth representations $\pi$ of $G(V)$ and $\pi'$ of $G(W)$, one has:
$$\dim\operatorname{Hom}_H(\pi\otimes\pi', \nu)\leq 1.$$
\end{theorem}

\begin{theorem}[Uniqueness of Fourier-Jacobi models]
Let $V$ be a linear space with a skew-symmetric or skew-Hermitian (including the case $\KK=\FF\times \FF$) form. Denote the respective symplectic or unitary group $G(V)$. Let $W$ be a subspace of even codimension on which the form is non-degenerate and so that $W^\perp$ is split. Let $H$ be the Fourier-Jacobi group corresponding to $W$, considered as a subgroup of $G(V)\times G(W)$. Let $\widetilde{H}$ be its appropriate double cover, and let $\nu$ be the representation of $\widetilde{H}$ constructed in \cite[Chapter 12]{GGP} (depending on some choices of characters). Take either $\pi$ to be an irreducible smooth representations of $\tG(V)$ (an appropriate double cover of $G(V)$) and $\pi'$ to be such a representation of $G(W)$, or the other way around, i.e. $\pi$ to be an irreducible smooth representation of $G(V)$ and $\pi'$ to be such a representation of an appropriate $\tG(W)$.
Then one has:
$$\dim\operatorname{Hom}_{\widetilde{H}}(\pi\otimes\pi', \nu)\leq 1.$$
\end{theorem}

\subsection{Comparison with previous works}
In \cite{GL case} the proof of a multiplicity one theorem for $GL_n$ in characteristic 0 is extended to include also positive odd characteristic.
The premise of this paper is to use these methods to extend the proof of additional multiplicity one theorems from characteristic 0 to positive odd characteristic. The proofs for characteristic 0 on which we base this paper are given in \cite{AGRS}, \cite{Wald} and \cite{Sun}. Let us give an overview of the methods and steps of this paper, explaining which ones are taken from \cite{AGRS}, \cite{Wald} and \cite{Sun}, which ones were introduced in \cite{GL case}, and which ones are new to this paper.

In section \ref{reformulations} we give reformulations of the problems in a way identical to the ones given in \cite{AGRS}, \cite{Wald}, and \cite{Sun}.

In section \ref{HC dec} we use a certain analogue of the Harish-Chandra descent method for positive characteristic, that gives weaker results than in the zero characteristic case. The entirety of this method as used in \cite{AGRS}, \cite{Wald} and \cite{Sun} fails over positive characteristic fields, due to non-seperable extensions, and in fact this is the crucial point in which these proofs fail for a positive characteristic.

In section \ref{lie-alg} we pass from the group to its Lie algebra using Cayley transorm. The difference from the analogous linearization in \cite{AGRS}, \cite{Wald} and \cite{Sun} is that in these papers linearization is done after using the method of Harish-Chandra descent to restrict the possible support to the unipotent cone, whereas we only have a weaker restriction on the support.

In section \ref{automorphisms} we adapt the main new ideas of \cite{GL case} to the  unitary, orthogonal, and symplectic settings, introducing a new family $\rho$ of automorphisms playing the same role as $\rho$ in \cite{GL case}.

Section \ref{stratification} uses the method of stratification to reduce the problem to a problem on a single orbit. The contents of this section are completely analogous to what is done in \cite{AGRS}, \cite{Wald} and \cite{Sun}, only without the restriction of nilpotency, which is not truly needed, as was the case in \cite{GL case}.

In section \ref{single stratum} we solve the previous problem on a single orbit by repeating the arguments and ideas used in \cite{AGRS}, \cite{Wald} and \cite{Sun}, sometimes giving slight generalizations of them.

The archimedean version of Theorems \ref{mult 1 for O} - \ref{J dist} can be found in \cite{SZ}.
Special cases of Theorems \ref{mult 1 for O} and \ref{G(V), G(W) dist} can be found in \cite{AGS}.

\subsection{Acknowledgements}
I would like to deeply thank my advisor, Dmitry Gourevitch, for helping me with my work on this paper, for exposing me to this fascinating area of mathematics, and for the guidance and support along the way. I would also like to thank him for the exceptional willingness to help.\\
I would also like to thank Guy Henniart for fruitful discussions he had with my advisor.\\
D.M. was partially supported by ERC StG grant 637912.

\section{Preliminaries and Notation}\label{preliminaries}
Most of this section is borrowed from the preliminaries sections of \cite{AG} and \cite{AGRS}, and also of \cite{GL case} (which was also mostly borrowed from the previous two).\\
Let $\FF$ be a local field of characteristic different from 2. Let $\KK$ be a field which is either equal to $\FF$ or a quadratic field extension of it. Let $\lambda\mapsto\bar{\lambda}$ be either the nontrivial automorphism of $\KK/\FF$ or the identity automorphism if $\KK = \FF$.
Let $V$ be a $\KK$-linear space of dimension $n$. Assume that we have on $V$ a non-degenerate sesquilinear form $B$ which is either symmetric, Hermitian, or symplectic (in the Hermitian case $\KK\neq \FF$ and in the other cases $\KK = \FF$).
Denote by $G=G(V)$ one of the groups $O(V), SO(V), U(V)$, or $Sp(V)$.  Denote by $\mathfrak{g}=\mathfrak{g}(V)$ the Lie algebra of $G$, which is either $\mathfrak{o}(V)$, $\mathfrak{u}(V)$, or $\mathfrak{sp}(V)$, i.e. linear transformations $A$ satisfying $A^*=-A$ with respect to the symmetric, Hermitian, or symplectic form.
In the $O, SO, U$ cases, assume we have $W\supseteq U$ of dimension $n+1$ with an extension of $<\cdot, \cdot>$ to a form on $W$ of the same type. In these cases we have also $G(W)$, and we may consider $G$ as a subgroup of it.

Let $\tG$ denote the subgroup of $\operatorname{Aut}_\FF(V)\times \{\pm 1\}$ consisting of all $(T, \delta)$ such that $<Tu, Tv> = <u, v>$ if $\delta = 1$ and $<Tu, Tv> = <v, u>$ if $\delta = -1$. In the case that $G = SO$ we also require that $\det T = \delta ^ {\floor{\frac{n+1}{2}}}$. This group contains $G$ as a subgroup of index $2$.
Denote by $\chi: \tG \to {\pm 1}$ the character $(T, \delta)\mapsto \delta$.
We have natural actions of $G$ on $G(W), G, \mathfrak{g}, V$ (by conjugation on all but $V$, on which we let $G$ act in the usual way).
This action extends to an action of $\tG$ by  $(T, \delta).A := TA^\delta T^{-1}$ on $G(W)$ and $G$, by $(T, \delta).A := \delta TAT^{-1}$ on $\mathfrak{g}$, and by $(T, \delta).v := \delta Tv$ on $V$.

\begin{notation}
Let $\Delta:G \to \KK[x]$ be the characteristic polynomial map. We shall also consider it as a map from $G\times V$, by first projecting onto $G$.
\end{notation}

We shall use the standard terminology of $l$-spaces introduced in
\cite{BZ}, section 1. We denote by $\Sc(Z)$ the space of Schwartz
functions on an $l$-space $Z$, and by $\Sc^*(Z)$ the space of
distributions on $Z$ equipped with the weak topology.

\begin{notation} [Fourier transform]
Let $W$ be a finite dimensional vector space over $\FF$ with a nondegenerate bilinear form $B$ on $W$. We denote by
$\Fou_B:\Sc^*(W) \to \Sc^*(W)$ the Fourier transform defined using $B$ and the self-dual Haar measure on $W$.
If $W$ is clear from the context, we sometimes omit it from the notation and denote  $\Fou=\Fou_W$.
\end{notation}

\begin{remark}
In the Hermitian case, we take for Fourier transform the $\FF$-bilinear form given by taking the trace of the Hermitian form.
\end{remark}

\begin{theorem}[Localization principle, see \cite{Ber}, section 1.4] \label{LocPrin}
Let  $q:Z \to T$ be a continuous map of $l$-spaces. We can consider $\Sc ^*(Z)$ as an $\Sc(T)$-module. Denote $Z_t:=
q^{-1}(t)$.
For any $M$ which is a closed $\Sc(T)$-submodule of $\Sc^*(Z)$,
$$M=\overline{\bigoplus_{t \in T} (M \cap \Sc^*(Z_t))}.$$
\end{theorem}
Informally, it means that in order to prove a certain property of
distributions on $Z$ it is enough to prove that distributions on
every fiber $Z_t$ have this property.
\begin{corollary} \label{LocPrinCor}
Let $q:Z \to T$ be a continuous map of $l$-spaces. Let an $l$-group
$H$ act on $Z$ preserving the fibers of $q$. Let $\mu$
be a character of $H$. Suppose that for any $t\in T$,
$\Sc^*(q^{-1}(t))^{H,\mu}=0$. Then $\Sc^*(Z)^{H,\mu}=0$.
\end{corollary}

\begin{corollary} \label{Product}
Let $H_i \subset \widetilde{H}_i$ be $l$-groups acting on $l$-spaces
$Z_i$ for $i=1, \ldots, k$. Suppose that
$\Sc^*(Z_i)^{H_i}=\Sc^*(Z_i)^{\widetilde{H}_i}$ for all $i$. Then
$\Sc^*(\prod Z_i)^{\prod H_i}=\Sc^*(\prod Z_i)^{\prod
\widetilde{H}_i}$.
\end{corollary}

\begin{theorem}[Frobenius descent, {\cite[section 1.5]{Ber}}] \label{Frob}
Let $H$ be a unimodular $l$-group acting on two $l$-spaces $E$ and $Z$, with the action on $Z$ being transitive. Suppose that we have an $H$-equivariant map $\varphi:E \to Z$. Let $x\in Z$ be a point with a unimodular stabilizer in $H$. Denote by $F$ the fiber of $x$ with respect to $\varphi$.
Then for any character $\mu$ of $H$ the following holds:
\begin{enumerate}[label=(\roman*)]
    \item There exists a canonical isomorphism $\mathrm{Fr}: \Sc^*(E)^{H,\mu} \to \Sc^*(F)^{\mathrm{Stab}_H(x),\mu}$.
    \item For any distribution $\xi \in \Sc^*(E)^{H,\mu}$, $\mathrm{Supp}(\mathrm{Fr}(\xi))=\mathrm{Supp}(\xi)\cap F$.
    \item Frobenious descent commutes with Fourier transform.
\end{enumerate}
\end{theorem}
To formulate (iii) explicitly, let $W$ be a finite dimensional linear space over $\FF$ with
a nondegenerate bilinear form $B$, and suppose $H$ acts on $W$ linearly
preserving $B$.
Then for any $\xi \in \Sc^*(Z\times W)^{H,\mu}$, we have
$\Fou_{B}(\mathrm{Fr}(\xi))=\mathrm{Fr}(\Fou_{B}(\xi))$, where $\mathrm{Fr}$ is taken with respect to the projection $Z \times W \to Z$.

\begin{remark}
Let $Z$ be an $l$-space and $Q\subset Z$ be a closed subset. We may
identify $\Sc^*(Q)$ with the space of all distributions on $Z$
supported on $Q$. In particular, we can restrict a distribution
$\xi$ to any open subset of the support of $\xi$.
\end{remark}

\begin{definition} \label{Regular element}
An element $A \in \mathrm{gl}(V)$ is said to be regular if its minimal polynomial is equal to its characteristic polynomial. In case that this polynomial $f$ of $A$ is a power of an irreducible polynomial we call $A$ a \textbf{minimal} regular element.
\end{definition}

\begin{theorem} [Rational Canonical Form] \label{Rational}
Any element $A\in \mathrm{gl}(V)$ can be represented as a direct sum of minimal regular elements. Moreover, the isomorphism classes of these elements are uniquely determined by the conjugacy class of $A$ inside $GL(V)$ and vice versa.
\end{theorem}

This form is called the rational canonical form of $A$.

\begin{definition}\label{f^*}
For any polynomial $f$ given by $f(x)=\sum_{i=0}^n a_i x^i$, define:
$$f^*(x)=\sum_{i=0}^n (-1)^i \bar{a_i} x^i,$$
$$f^\dagger(x)=\sum_{i=0}^n \overline{a_{n-i}} x^i.$$
\end{definition}
\begin{remark}
For any $A\in \mathfrak{g}$ (respectively $A\in G(V)$), we have $f(A)^* = f^*(A)$ (respectively $f(A)^\dagger=A^{-n}f^\dagger(A)$).
\end{remark}

\begin{lemma}\label{eigenvalues perp}
Let $A\in \mathfrak{g}$ ($A\in G$).
Let $\left(f_i\right)_{i\in I}$ be the different irreducible factors in the characteristic polynomial of $A$. Let $\left(V_i\right)_{i\in I}$ be the generalized eigenspace associated with each.
Take $i,j \in I$, not necessarily different. If $f_i^*\neq \pm f_j$ (respectively $f_i^\dagger\nmid f_j$), then $V_i\perp V_j$.
\end{lemma}
\begin{proof}
We give the proof in the $A\in\mathfrak{g}$ case (the $A\in G$ case is analogous). Take $u\in V_i, v\in V_j$. Then for $k$ large enough:
$$0=<f_i(A)^k u,v>=<u,f_i^*(A)^k v>.$$
Thus $V_i \perp f_i^*(A)^k V_j$, but if $f_i^* \neq\pm f_j$, then they are coprime to each other, and so $f_i^*(A)^k V_j=V_j$.
\end{proof}

\begin{definition} \label{block types}
\begin{enumerate}
    \item An operator $A\in\mathfrak{g}$ will be called a simple split operator (or block) if the following conditions hold:
    \begin{itemize}
        \item There is a possibly non-orthogonal decomposition $V=V'\oplus \overline{V'^*}$.
        \item $V'$ and $\overline{V'^*}$ are isotropic and the sesquilinear form $B$ induces the natural pairing between them.
        \item The action of $A$ on $V'\oplus \overline{V'^*}$ decomposes as $A'\oplus A''$.
        \item $<A'u,v>=<u,-A''v>$ for any $u\in V', v\in \overline{V'^*}$. 
        \item $A'$ (and so also $A''$) is a minimal regular operator (see Definition \ref{Regular element}).
        \item The irreducible factor $f$ of the minimal polynomial of $A'$ is not equal to $f^*$.
    \end{itemize}
    \item An operator $A\in\mathfrak{g}$ will be called a simple non-split operator (or block) if \begin{itemize}
        \item It is a minimal regular operator.
        \item Its characteristic polynomial is not equal to $x^d$ with $d$ even if $\mathfrak{g} = \mathfrak{o}$, and it is not equal to $x^d$ with $d$ odd if $\mathfrak{g} = \mathfrak{sp}$.
    \end{itemize}
    \item An operator $A\in\mathfrak{o}$ will be called a simple even nilpotent operator (or block) if the following conditions hold:
    \begin{itemize}
        \item Its minimal polynomial is $x^d$ for some even $d$.
        \item $V$ has a basis of the form $e, Ae, \dots, A^{d-1}e, f, Af, \dots, A^{d-1}f$.
        \item For all $i, j$ we have $<A^ie, A^je> = <A^if, A^jf> = 0$.
        \item For all $i, j$ we have 
        $<A^ie, A^jf> =
        \begin{cases}
        (-1)^{j} & \mathrm{if}\ i + j = d - 1\\
        0  & \mathrm{otherwise}
        \end{cases}
        $.
    \end{itemize}
    \item An operator $A\in\mathfrak{sp}$ will be called a simple odd nilpotent operator (or block) if the following conditions hold:
    \begin{itemize}
        \item Its minimal polynomial is $x^d$ for some odd $d$.
        \item $V$ has a basis of the form $e, Ae, \dots, A^{d-1}e, f, Af, \dots, A^{d-1}f$.
        \item For all $i, j$ we have $<A^ie, A^je> = <A^if, A^jf> = 0$.
        \item For all $i, j$ we have 
        $<A^ie, A^jf> =
        \begin{cases}
        (-1)^{j} & \mathrm{if}\ i + j = d - 1\\
        0  & \mathrm{otherwise}
        \end{cases}
        $.
    \end{itemize}
\end{enumerate}
\end{definition}


The following useful proposition will be proved in Appendix \ref{classifications}.
\begin{proposition}\label{full classification}
Each $A\in \mathfrak{u}$ decomposes as an orthogonal sum of simple split blocks and simple non-split blocks. Each $A\in \mathfrak{o}$ decomposes as an orthogonal sum of simple split blocks, simple non-split blocks, and simple even nilpotent blocks. Each $A\in \mathfrak{sp}$ decomposes as an orthogonal sum of simple split blocks, simple non-split blocks, and simple odd nilpotent blocks.
\end{proposition}

\begin{remark}
The contents of Proposition \ref{full classification} are contained in known papers and books such as \cite{BCCISS} and \cite{Wall}. However, for clarity and completeness, we formulated only the propositions we need and give short proofs of them in Appendix \ref{classifications}.
\end{remark}

\begin{proposition}[{\cite[Appendix A]{GL case}}]\label{matrix determinant}
Let $V$ be a linear space of finite dimension $n$ over a field $\KK$, $A\in \mathrm{gl} (V)$, $v\in V$, and $\phi \in V^*$. The following are equivalent.
\begin{enumerate}
    \item $\forall k\geq 0,\ \phi A^k v=0$.
    \item For all $\lambda\in\KK$, $\mathrm{ch}(A+\lambda v\otimes \phi)=\mathrm{ch}(A)$.
    \item There exists $\lambda \in \FF ^ \times$ such that $\mathrm{ch}(A+\lambda v\otimes \phi)=\mathrm{ch}(A)$.
    \item $\frac{\partial}{\partial \lambda}\mathrm{ch}(A+\lambda v\otimes\phi)|_{\lambda=0}=0.$
\end{enumerate}
\end{proposition}

Let $V$ be a vector space over a local field $\FF$ of characteristic different from $2$. Let $\operatorname{GL}(V)$ act on $\mathfrak{gl}(V) \times V \times V^*$ by conjugation, and the natural actions on $V, V^*$. Consider also the transposition involution, which involves a choice of an isomorphism $t:V\to V^*$, and sends $(A, v, \phi)$ to $A^t, \phi^t, v^t$. As immediate corollaries of \cite[Theorem 3.1]{GL case} (and of the proof that it implies Theorem 1.1 of the same paper) we have:
\begin{theorem} \label{GL theorem}
Any $\operatorname{GL}(V)$-invariant distribution on $GL(V) \times V \times V^*$ is also invariant to transposition.
\end{theorem}
\begin{theorem} \label{GL theorem Lie-alg}
Any $\operatorname{GL}(V)$-invariant distribution on $\mathfrak{gl}(V) \times V \times V^*$ is also invariant to transposition.
\end{theorem}
\section{Reformulations of the problem}\label{reformulations}
Let $V, G, \tG, \chi$ be as in Section \ref{preliminaries}.
Both Theorem \ref{G(V), G(W) dist} and Theorem \ref{J dist} follow from the following theorem:
\begin{theorem} \label{main goal}
Any $(\tG, \chi)$-equivariant distribution on $G\times V$ is 0.
\end{theorem}
\begin{proof}[Proof that Theorem \ref{main goal} implies Theorem \ref{G(V), G(W) dist}]
This proof is the same as the proof of \cite[Proposition 5.1]{AGRS}.
Use the notations given in the introduction (e.g. $V,W,G(V),\tG(V)$).
We actually prove that Theorem \ref{main goal} for $W$ implies Theorem \ref{G(V), G(W) dist} for $V$.
The idea is to consider the set
$$Y:=\{w\in W | <w,w> = <v_{n+1},v_{n+1}> \}.$$
and use Frobenius descent (Theorem \ref{Frob}) on the projection $G(W)\times Y \to Y$. The group $\tG(W)$ Acts on $Y$ with centralizer $\tG(V)$. The fiber of the projection is $G(W)$. So we get a bijection between $(\tG(V), \chi)$-equivariant distributions on $G(W)$ and $(\tG(W), \chi)$-equivariant distributions on $G(W)\times Y$. The latter of which has no non-zero elements by the assumption on $W$. Thus any $(\tG(V), \chi)$-equivariant distribution on $G(W)$ is 0, which immediately implies Theorem \ref{G(V), G(W) dist}.
\end{proof}
\begin{proof}[Proof that Theorem \ref{main goal} implies Theorem \ref{J dist}]
The group $S$ in the formulation of Theorem \ref{J dist} decomposes as a product of groups of types $O, SO, U, Sp$, with $\sigma$ a product of appropriate anti-involutions. By Corollary \ref{LocPrinCor} of the localization principle, it follows that Theorem \ref{main goal} implies an analogous claim for $S$, which can be seen to easily imply Theorem \ref{J dist}.
\end{proof}
The proof of the theorem is by induction on $\dim V$, proving simultaneously the following theorem:
\begin{theorem}\label{main goal for Lie Alg}
Any $(\tG, \chi)$-equivariant distribution on $\mathfrak{g}\times V$ is 0.
\end{theorem}
For $n=0$ Theorems \ref{main goal} and \ref{main goal for Lie Alg} are trivial.

\section{Harish-Chandra descent}\label{HC dec}
In this section we use the technique of Harish-Chandra descent to restrict the support of an equivariant distribution as discussed in Theorem \ref{main goal} and Theorem \ref{main goal for Lie Alg}. Assume for this section by induction Theorem \ref{main goal} and Theorem \ref{main goal for Lie Alg} for smaller dimensions, over all finite field extensions of $\KK$.\\
Let $(A, v)$ be a point in the support of a $(\tG, \chi)$-equivariant distribution either on $G\times V$ (the group case) or on $\mathfrak{g}\times V$ (the Lie-algebra case). Let $g(X)$ be the characteristic polynomial of $A$. Consider also the characteristic polynomial map $\Delta:G\times V \to \KK[x]$ (or $\Delta:\mathfrak{g}\times V \to \KK[x]$ in the Lie algebra case). Note that $g\mid g^\dagger$ in the group case, and $g=\pm g^*$ in the Lie algebra case (recall definition \ref{f^*} of $g^\dagger$ and $g^*$).
\begin{theorem}\label{HC dec 1}
Unless we are in the group case and $G = SO$, the polynomial $g$ cannot be factorized into two coprime factors $g_1, g_2$ satisfying $g_1\mid g_1^\dagger$ and $g_2\mid g_2^\dagger$ (respectively $g_1 = \pm g_1^*$ and $g_2 = \pm g_2^*$ in the Lie algebra case).
In the case $G=SO$, it is still true that it is impossible for $g$ to be divisible by both $x-1$ and $x+1$.
\end{theorem}

\begin{proof}
We give the proof for the group case and for the Lie algebra case simultaneously. By the localization principle (Corollary \ref{LocPrinCor}) it is enough to show that there is no $(\tG, \chi)$-equivariant distribution $\xi$ on any of the fibers of $\Delta$ which is above a polynomial not satisfying the condition we gave on $g$.
Let $F$ be such a fiber lying above a polynomial $g(x) = g_1(x)g_2(x)$ with $g_1, g_2$ coprime and of positive degree, satisfying $g_1\mid g_1^\dagger$ and $g_2\mid g_2^\dagger$ ($g_1 = \pm g_1^*$ and $g_2 = \pm g_2^*$ in the Lie algebra case).
If we are in the group case and $G=SO$, we further assume that $g_1(x) = (x-1)^k$ for some $k>0$.
Let $d_1, d_2$ be the degrees of $g_1, g_2$.
Given $A$ with characteristic polynomial $g(x)$, one may consider $V_1, V_2$, its generalized eigenspaces associated with $g_1(x), g_2(x)$ respectively.
By Lemma \ref{eigenvalues perp}, $V_1, V_2$ are perpendicular to each other.
Consider
$$\Lambda=\{V_1, V_2\subset V|V=V_1\oplus V_2,\ V_1\perp V_2,\ \dim V_i=d_i\},$$
to be the space of decompositions of $V$ as an orthogonal sum $V_1\oplus V_2$ to subspaces of dimensions $d_1, d_2$.
There is a natural $\tG$-equivariant map $\rho: F\to \Lambda$.
Consider the stratification on $\Lambda$ given by $G$-orbits. Note that these are the same as $\tG$-orbits. To show that this is indeed a stratification we must show that there are finitely many $G$-orbits. Recall that there are finitely many isomorphism classes of sesquiliner forms of the same type as $B$ (symmetric, Hermitian, or symplectic) on a $\KK$-vector space of a given dimension. If two elements in $\Lambda$ share the isomorphism classes of the restrictions of $B$ to $V_1, V_2$, then these isomorphisms can be extended orthogonally to an element of $G$ (in the case $G=SO$, it will only be an element of $O$. However it is enough to prove that there are finitely many $O$ orbits). This implies that the two elements we had in $\Lambda$ are in the same $G$-orbit ($O$-orbit if $G=SO$). It follows that there are indeed finitely many $G$-orbits, and so partition into orbits is a stratification.

Let $S$ be the union of strata intersecting $\rho(\supp(\xi))$, and let $\Omega$ be a stratum of the largest dimension in it (we assume by contradiction that $\xi\neq 0$, i.e. $S$ is non-empty). It is open in $S$, and so we may restrict $\xi$ to $\rho^{-1}(\Omega)$. Since $\Omega\ss S$, this restriction is not the zero distribution.

For the following assume that we are not in the groups case where $G=SO$.
The action of $\tG$ on $\Omega$ is transitive by definition, and the stabilizer of a point in $\Omega$ (Call it $H$) is a subgroup of index $2$ of $\tG(V_1) \times \tG(V_2)$, which is a unimodular group (thus it is also unimodular). $H$ also contains $G(V_1)\times G(V_2)$ as a subgroup of index $2$.
Using Frobenius descent (Theorem \ref{Frob}) on $\xi$, we get an $(H, \chi)$-equivariant distribution on the fiber, which is a closed subspace of $(G(V_1)\times V_1) \times (G(V_2)\times V_2)$. In particular this distribution is $G(V_1)\times G(V_2)$-invariant. Hence this distribution is $\tG(V_1)\times \tG(V_2)$-invariant by the induction hypothesis and Corollary \ref{Product} to the Localization Principle. In particular it is also $H$-invariant, thus it is 0, in contradiction to our assumption.

In the case $G=SO$, we have a similar situation.
The action of $\tSO(V)$ on $\Omega$ is transitive, and the stabilizer of a point in $\Omega$, which we will call $H$, is a unimodular subgroup of index $4$ inside $\tO(V_1)\times \tO(V_2)$. This group $H$ contains $SO(V_1)\times SO(V_2)$ as a subgroup of index $4$, on which the character $\chi$ is trivial. Since the determinant of an operator acting on $V_1$ with characteristic polynomial $g_1(x)=(x-1)^k$ and on $V_2$ with characteristic polynomial $g_2(x)$ is 1, we get that $g_2(0) = (-1)^{\dim V_2}$. If $\dim V_2$ was odd, it would imply that $g_2^\dagger = -g_2$, and in particular $g_2(1) = 0$. By assumption this is not the case, and so we have that $\dim V_2$ must be even. It follows that
$$(-1)^{\lfloor\frac{\dim V_1 + 1}{2}\rfloor}(-1)^{\lfloor\frac{\dim V_2 + 1}{2}\rfloor} = (-1)^{\lfloor\frac{\dim V_1 + 1}{2}\rfloor + \frac{\dim V_2}{2}} = (-1)^{\lfloor\frac{\dim V + 1}{2}\rfloor}.$$
Take elements $(g_1, -1) \in\tSO(V_1)\setminus SO(V_1)$ and $(g_2, -1)\in\tSO(V_2)\setminus SO(V_2)$.
From the above, it follows that $\gamma:=(g_1\oplus g_2, -1)$ is an element of $(\tSO(V_1)\times \tSO(V_2))\cap H$, on which $\chi$ gives $-1$.
The fiber of $\rho$ above $(V_1, V_2)$ is $SO(V_1)\times SO(V_2)$, because any element in it acts with characteristic polynomial $(x-1)^k$ on $V_1$, and thus has determinant $1$ when restricted to it. It follows that the restriction to $V_2$ also has determinant $1$.
As before, we get that any $(H,\chi)$-equivariant distribution on $SO(V_1)\times SO(V_2)$ is invariant to $\tSO(V_1)\times \tSO(V_2)$ (using the localization principle and the induction hypothesis).
In particular it is $\gamma$-invariant, thus it is $0$, since $\chi(\gamma)=-1$. Using Frobenius descent, we get that this implies $\xi = 0$, giving a contradiction.
\end{proof}

We give the following theorem only in the Lie algebra case as this is what will be used. However, it also holds in the group case, with the same proof.



\begin{theorem}\label{HC dec 2}
Assume we are in the Lie algebra case. The polynomial $g$ must be a power of an irreducible polynomial.
\end{theorem}
\begin{proof}
Again we use the localization principle. Let $F$ be the fiber above a polynomial of the form $g(x) = g_1(x)g_2(x)$ with $g_2 = \pm g_1^*$, and the two are coprime to each other. (By Theorem \ref{HC dec 1} it is enough to consider this case).
Given $A$ with characteristic polynomial $g(x)$, one may consider $V_1, V_2$, its generalized eigenspaces associated with $g_1(x), g_2(x)$ respectively.
By Lemma \ref{eigenvalues perp}, $V_1, V_2$ are both isotropic.
Consider
$$\Lambda=\{V_1, V_2\subset V|V=V_1\oplus V_2,\ B|_{V_1} = 0,\ B|_{V_2} = 0,\ \dim V_1=\dim V_2 = \frac{\dim V}{2}\},$$
There is a natural $\tG$-equivariant map $\rho: F\to \Lambda$.

To see that $G$ acts transitively on $\Lambda$, take $(V_1, V_2), (V_1', V_2')\in \Lambda$. Choose arbitrary bases $E_1, E_1'$ of $V_1, V_1'$. We may take $E_2$ to be the basis of $V_2$ dual to $E_1$ with respect to the pairing between $V_1, V_2$ induced by $B$. Similarly we may take $E_2'$. The linear transformation which sends $E_1$ to $E_1'$ and $E_2$ to $E_2'$ preserves $B$, and thus it is an element of $G$.
So the actions of both $G$ and $\tG$ on $\Lambda$ are transitive, and the stabilizer inside $\tG$ of a point in $\Lambda$ is isomorphic to $\widetilde{\operatorname{GL}}(V_1)$, which is a unimodular group.
Using Frobenius descent (Theorem \ref{Frob}) on $\xi$, we get a $(\widetilde{\operatorname{GL}}(V_1), \chi)$-equivariant distribution on the fiber, which is isomorphic to
$$\mathfrak{gl}(V_1)\times V_1 \times V_2\cong \mathfrak{gl}(V_1)\times V_1\times V_1^*.$$
By Theorem \ref{GL theorem Lie-alg} this distribution must be equal $0$, and so is the original one.
\end{proof}
We formulate the next theorem only for the Lie algebra case, and $\mathfrak{g}=\mathfrak{sp}$, although again it is also true for all the other cases.
\begin{theorem}\label{HC dec 3}
Consider the Lie algebra case of $\mathfrak{g} = \mathfrak{sp}$. In this case, the irreducible factor of $g$ is either linear or inseparable.
\end{theorem}
\begin{proof}
Again we use the localization principle. Let $F$ be the fiber above a polynomial $g(x) = f(x)^s$ with $f$ irreducible, seperable, of degree $d > 1$, and satisfying $f^* \neq \pm f$. Given $A$ with characteristic polynomial $g(x)$, we may consider its additive Jordan decomposition into semi-simple and unipotent parts, $A_s$ and $A_u$ (that is in virtue of the characteristic polynomial being seperable). Let $F_s$ be the space of possible $A_s$-s, that is the space of semisimple elements of $\mathfrak{g}$ with characteristic polynomial $g(x)$. We have a $\tG$-equivariant map $\theta: F\to F_s$.
By \cite[E, IV, Section 2]{BCCISS}, $F_s$ is a disjoint union of finitely many $G$-orbits, all of the same dimension. By \cite[Chapter 4, Proposition 1.2.]{MVW}, for each point $A\in F_s$, there is an element in $\tG\setminus G$ which centralizes it (see the details of this implication in Lemma \ref{tc}). Thus $G$-orbits is $F_s$ are $\tG$-invariant.
So it is enough to show that for any orbit $\mathfrak{O}\ss F_s$, any $(\tG, \chi)$-equivariant distribution on $\theta^{-1}(\mathfrak{O})$ is $0$. By Frobenius descent, this is equivalent to showing that for some $A\in \mathfrak{O}$, any $(\tG_A, \chi)$-equivariant distribution on $\theta^{-1}(A)$ is $0$ ($\tG_A$ being the stabilizer of $A$ in $\tG$).

In order to prove this, let us describe the stabilizer of a point $A$ in $F_s$.
Let $m:=\FF[T] / f(T)$. We define an $\FF$-linear involution of $m$ by $\sigma:h(T)\mapsto h^*(T) = h(-T)$ (that is the same as saying $T\mapsto -T$). Let $m_0$ be the fixed subfield of this involution. It is a subfield of $m$ of index $2$.
Fix a non-zero $\FF$-linear functional $\ell:m\to \FF$ which satisfies that $\ell(\sigma h) = \ell(h)$ for all $h\in \FF$. Any other $\FF$-linear functional can be written as $h\mapsto \ell(\lambda h)$ for some unique $\lambda \in m$.
Being semi-simple, $f(A)$ must be equal 0. Thus $V$ can be given the structure of a linear space over $m$, by $hv:=h(A)v$. Given $v, v'\in V$, the map $h\mapsto <h(A)v, v'>$ is an $\FF$-linear functional $m\to \FF$, and so can be written as $<h(A)v, v'> = \ell(S(v, v')h)$ for some $S(v, v')\in m$. One may check that $S(v, v')$ is $m$-sesquilinear (with respect to $\sigma$) considering $V$ as a linear space over $m$ by $h\cdot v:=h(A)v$. The form $S$ is also non-degenerate, and satisfies $S(v',v) = -\sigma S(v, v')$.
Fix $a\in m$ such that $\sigma a = -a$ (e.g. $a = T\in \FF[T] / f(T)$). Then it follows from the above that $aS(\cdot, \cdot)$ is a non-degenerate Hermitian form on $V_m$ (with respect to the involution $\sigma$), where $V_m$ is $V$ as a linear space over $m$.
To say that a linear automorphism of $V$ commutes with $A$ is to say that it is $m$ linear, and for such an automorphism to say that it is in $G(V)$ is to say that it preserves $aS$. Thus we have that the centralizer of $A$ in $G(V)$ can be described as $U(V_m)$. Moreover, the stabilizer of $A$ in $\tG(V)$ can be described as $\tU(V_m)$. Also the centralizer of $A$ inside $\mathfrak{g}(v)$ can be described as $\mathfrak{u}(V_m)$.

Recall that we need to show that any $(\tU(V_m), \chi)$-equivariant distribution on $\theta^{-1}(A)$ is $0$. The space $\theta^{-1}(A)$ is identified with $\mathfrak{u}_n(V_m)\times V_m$, $\mathfrak{u}_n(V_m)$ being the space of nilpotent elements in $\mathfrak{u}(V_m)$. This in turn is a closed subspace of $\mathfrak{u}(V_m)\times V$. Thus our claim follows from the fact that any $(\tU(V_m), \chi)$-equivariant distribution on $\mathfrak{u}(V_m)$ is $0$, which follows from our induction hypothesis of \ref{main goal}, as $\dim V_m < n$.
\end{proof}

\section{Separation of 1,-1 as eigenvalues, and passage to the Lie algebra}\label{lie-alg}
In this section we pass from the group case to the Lie algebra case, showing that Theorem \ref{main goal for Lie Alg} implies theorem \ref{main goal}.
\begin{definition}
Let $G_{(1)}$ be the open subset of $G$ consisting of elements of which $1$ is not an eigenvalue. Similarly, define $G_{(-1)}$ to be the open subset of elements of which $-1$ is not an eigenvalue. Define also $\Xi:=G \setminus(G_{(1)} \cup G_{(-1)})$, i.e. elements of which both $1,-1$ are eigenvalues.
\end{definition}

The following proposition is an immediate corollary of Theorem \ref{HC dec 1}:
\begin{proposition}
To prove Theorem \ref{main goal} it is enough to show that any $(\tG, \chi)$-equivariant distribution on $G_{(\pm 1)} \times V$ is 0.
\end{proposition}

\begin{definition}[Cayley transform]
Define $C_1:G_{(1)}(V)\to \mathfrak{g}$ by $C_1(A)=\frac{I+A}{I-A}$. Similarly define $C_{-1}:G_{(-1)}(V)\to \mathfrak{g}$ by $C_{-1}(A)=\frac{I-A}{I+A}$.
\end{definition}
This definition makes sense since for $A\in G_{(1)}(V)$:
$$\left(\frac{I+A}{I-A} \right)^* = \frac{I+A^*}{I-A^*} = \frac{I+A^{-1}}{I-A^{-1}} = \frac{A+I}{A-I} = -\frac{I+A}{I-A},$$
and similarly for $C_{-1}$. 
\begin{definition}
Let $\mathfrak{g}_0$ be the subspace of $\mathfrak{g}$ not having $\pm 1$ as eigenvalues.
\end{definition}
\begin{proposition}
The maps $C_{\pm 1}$ are $\tG$-homeomorphisms from $G_{(\pm 1)}(V)$ (respectively) to $\mathfrak{g}_0$, unless we are in the case $G=SO$, considering $C_1$, and $\dim V$ is even. In this case $SO_{(1)}(V)=\emptyset$.
\end{proposition}
\begin{proof}
First, exclude the case $G=SO$.
We will give the proof for $C_1$. The proof for $C_{-1}$ is similar.
First notice that indeed $C_1(A)$ does not have $\pm 1$ as eigenvalues. If it did have, then $(I+A)v=\pm (I-A)v$ which leads to either $v=0$ or $Av=0$, and $A$ is invertible.
Second, we construct an inverse, $B \mapsto \frac{B-I}{B+I}$:
One can see that the inverse map is indeed into $G_{(1)}$, as:
$$\left(\frac{B-I}{B+I}\right)^*=\frac{B^*-I}{B^*+I}=\frac{-B-I}{-B+I}=\frac{B+I}{B-I}=\left(\frac{B-I}{B+I}\right)^{-1}$$
and also $\frac{B-I}{B+I}$ cannot have $1$ as an eigenvalue, as then $(B-I)v=(B+I)v$, hence $v=0$.
To see that we indeed constructed an inverse map:
$$\frac{I+\frac{B-I}{B+I}}{I-\frac{B-I}{B+I}}=\frac{B+I+B-I}{B+I-(B-I)}=\frac{2B}{2I}=B.$$
$$\frac{\frac{I+A}{I-A}-I}{\frac{I+A}{I-A}+I}=\frac{I+A-(I-A)}{I+A+I-A}=\frac{2A}{2I}=A.$$
There are similar arguments for $C_{-1}$ showing it is a $\tG$-isomorphism to $\mathfrak{g}_0$.\\
For the case $G=SO$, it simply holds that $SO_{(-1)} = O_{(-1)}$, $SO_{(1)} = O_{(1)}$ when $\dim V$ is even, and $SO_{(1)} = \emptyset$ when $\dim V$ is odd.
\end{proof}
\begin{proposition}
To prove Theorem \ref{main goal} it suffices to show that any $(\tG, \chi)$-equivariant distribution on $\mathfrak{g}_0 \times V$ is 0.
\end{proposition}
\begin{proof}
Use $C_{\pm 1}$.
\end{proof}

\begin{proposition} \label{linearization}
Theorem \ref{main goal} follows from Theorem \ref{main goal for Lie Alg}.
\end{proposition}
\begin{proof}
Take $\xi$ to be a $(\tG, \chi)$-equivariant distribution on $\mathfrak{g}_0 \times V$. Assume by contradiction that it is not $0$, and let $(A_0,v_0)$ be a point in its support. Let $t=\det((A_0-I)(A_0+I))\neq 0$. One can choose $f\in \Sc(\KK)$ s.t. $f(t)\neq 0,f(0)=0$. Note that $\mathfrak{g}_0$ is an open subset of $\mathfrak{g}$, and $g(A):=f(\det((A-I)(A+I)))$ is a locally constant function, compactly supported inside $\mathfrak{g}_0$. Thus we can extend $g\cdot \xi$ to a $(\tG, \chi)$-equivariant distribution on $\mathfrak{g}\times V$ with $(A_0,v_0)$ in its support. In particular this distibution is not 0, which creates a contradiction to our assumption.
\end{proof}

\section{An important lemma and Automorphisms}\label{automorphisms}

\begin{lemma}\label{v perp}
Any $(\tG, \chi)$-equivariant distribution on $\mathfrak{g} \times V$ is supported on $\mathfrak{g}\times \Gamma$, where $\Gamma:=\{v\in V | <v,v>=0\}$.
\end{lemma}
\begin{proof}
We do not consider in the following the case $\mathfrak{g}=\mathfrak{sp}$, as in this case $\Gamma=V$ and there is nothing to prove.
This proof is the same as the proof of \cite[Proposition 5.2]{AGRS}.
The idea is to consider the map $\mathfrak{g}\times V\to \KK$ given by $(A,v)\mapsto <v,v>$, and apply the localization principle (Corollary \ref{LocPrinCor}) to it to restrict to a fiber. Then apply Frobenius descent (Theorem \ref{Frob}) on the projection on the second coordinate, to reach a point where it is enough to show that any $(\tG(V'), \chi)$-equivariant distribution on $\mathfrak{g}(V')$ is 0, for some subspace $V'\ss V$ of codimension 1.
We have a decomposition $\mathfrak{g} = \mathfrak{g}(V')\oplus V'\oplus E$, with $E$ being either a $0$ or $1$ dimensional vector space over $\FF$ with trivial $\tG(V')$-action, and so we can use the induction hypothesis to finish.
\end{proof}
Denote by $\phi_v$ the linear transformation $u\mapsto <u,v>v$.

The following definition will be relevant for the cases of $\mathfrak{u}$ and $\mathfrak{sp}$:
\begin{definition}
For any $\lambda\in \KK$, requiring $\bar{\lambda}=-\lambda$ if $\mathfrak{g} = \mathfrak{u}$, we define an automorphism of $\mathfrak{g}$ by
$\nu_{\lambda}(A,v):= (A + \lambda \phi_v, v)$. This is an automorphism of $\mathfrak{g}$ as a space with a $\tG$ action.
\end{definition}

The following definition will be relevant only for case of $\mathfrak{o}$:
\begin{definition}
For any $\lambda\in \FF$, define an automorphism of $\mathfrak{g}\times V$ by
$$\mu_{\lambda}(A,v):=(A+\lambda A\phi_v + \lambda \phi_v A, v).$$
This is an automorphism of $\mathfrak{g}\times V$ as a space with a $\tG$ action.
\end{definition}

Fix a fiber $F$ of $\Delta:\mathfrak{g}\times V\to \KK[x]$ at a polynomial $f$. Recall that we must have $f^*(x)=(-1)^n f(x)$. Choose a polynomial $g\in \KK[x]$ coprime to $f$ that also satisfies $g^*(x) = g(x) \mod f(x)$.
Then we can define:
\begin{definition}
Define an automorphism of $F$ by $\rho_g(A,v)=(A,g(A)v)$.
\end{definition}
To show that it is invertible, notice that there is an 'inverse' polynomial $g^{-1}$ such that $gg^{-1}=1\mod f$. It also satisfies $(g^{-1})^*(x)= g^{-1}(x) \mod f(x)$, as for some polynomial $a$:

\begin{align*}
    1&=((g^{-1}g +af))^*(x)=(g^{-1})^*(x)g^*(x)+a^*(x)f^*(x)=\\
    &=(g^{-1})^*(x)g^*(x) + (-1)^n a^*(x)f(x).
\end{align*}
The last being equal to $(g^{-1})^*(x)g(x)$ modulo $f(x)$.
This implies that we have $\rho_{g^{-1}}$ which is inverse to $\rho_g$.
To show that $\rho_h$ is commutes with the action of $\tG$, the only nontrivial part is to show that it commutes with the action of an element $x\in \tG \setminus G$. Consider $x$ as an element of $\operatorname{End}_\FF(V)$ satisfying $<xu, xw> = <w,u>$ for any $u,w\in V$. In particular $x$ satisfies $axu = x\bar{a}u$ for any $a\in \KK, u\in V$.
To show commutation, we need to show that $-xg(A)v=g(-xAx^{-1})(-xv)$. This is true as:
$$g(-xAx^{-1})(-xv)=-xg^*(A)x^{-1}(xv)=-xg^*(A)v=-xg(A)v.$$
For the last equation we used the condition imposed on $g$, and the fact that $f(A)=0$.

Thus we get that $\rho_g$ is a $\tG$-automorphism of $F$.

In the case $\mathfrak{g}=\mathfrak{sp}$ we give the following lemma by using the automorphisms $\nu_\lambda$ to amplify the restriction of Theorem \ref{HC dec 3}:
\begin{lemma}\label{Av perp v}
Assume $\mathfrak{g} = \mathfrak{sp}$. Then any $(\tG, \chi)$-equivariant distribution on $\mathfrak{g} \times V$ is supported on $\mathfrak{g}\times \Gamma_1$, where $\Gamma_1:=\{v\in V | <Av,v>=0\}$.
\end{lemma}
\begin{proof}
Given $A\in \operatorname{End}(V)$ write the characteristic polynomial of $A$ as
$$x^n + c_1(A)x^{n-1} + c_2(A)x^{n-2} + \dots$$
with the convention that $c_0(A) = 1$.
Let $(A, v)$ be a point in the support of a $(\tG, \chi)$-equivariant distribution on $\mathfrak{g} \times V$. From Theorem \ref{HC dec 3} we deduce that we have $c_1(A) = n\alpha$ and $c_2(A) = \binom{n}{2}\alpha^2$ for some $\alpha \in \FF$ (if the characteristic polynomial is a power of a linear polynomial this is clear, and if it is a power of an inseperable polynomial then we indeed have $c_1(A) = c_2(A) = 0$). Note that using this condition, $c_1(A)$ determines uniquely $c_2(A)$. Now we may translate by $\nu_\lambda$ and then apply the above condition to get $c_1(A + \lambda \phi_v) = n\beta, c_2(A + \lambda \phi_v) = \binom{n}{2}\beta^2$. However by \cite[Theorem A.2]{GL case} we have $$c_1(A + \lambda \phi_v) = c_1(A) - <v,v>$$ and $$c_2(A + \lambda \phi_v) = c_2(A) - c_1(A)<v,v> - <Av,v>.$$ Since $<v,v> = 0$, this is saying that $c_1(A + \lambda \phi_v) = c_1(A)$, hence $\beta = \alpha$ and so also $c_2(A + \lambda \phi_v) = c_2(A)$, and from the second equation $c_2(A + \lambda \phi_v) = c_2(A) - <Av,v>$. Thus $<Av, v> = 0$.
\end{proof}
\begin{definition}
For any $A\in \mathfrak{g}$, denote
$$R_A:=\{v\in V | \forall k\geq 0, <A^k v,v>=0\}.$$
Denote also
$$R=\{(A,v)\in\mathfrak{g}\times V|v\in R_A\}.$$
\end{definition}
\begin{proposition}\label{inside R}
Any $(\tG, \chi)$-equivariant distribution on $\mathfrak{g}\times V$ is supported on $R$.
\end{proposition}
\begin{proof}
By the Localization principle (Corollary \ref{LocPrinCor}), it is enough to show that any $(\tG, \chi)$-equivariant distribution on a fiber $F$ of $\Delta$ at a polynomial $f$ is supported on $R\cap F$ (note that $R$ is $\tG$-invariant).
Let $\xi$ be such a distribution, and let $(A,v)$ be a point in $\mathrm{supp}(\xi)$.
Let us start with the case of $\mathfrak{u}$.

Choose $\omega\in\KK^\times$ with $\bar{\omega}=-\omega$.
Let $g\in\FF[x]$ and consider $g_1(x)=g(x^2)$ and $g_2(x)=\omega x g(x^2)$. They satisfy $\overline{g_1}(-x)=g_1(x),\overline{g_2}(-x)=g_2(x)$. Choose $g$ such that $g_1$ will be coprime to $f$.
We can apply $\rho_{g_1}$ to $\xi$ and extend back to $\mathfrak{g}\times V$ to get that by Lemma \ref{v perp}, $<g_1(A)v,g_1(A)v>=0$. We know this for a Zariski dense subset of polynomials $g\in \FF[x]$, and so for all $g\in \FF[x]$. The same goes for $g_2$. So in particular:
$$<A^{2k}v,v>=\frac{<(A^{2k}+I)v,(A^{2k}+I)v>-<A^{2k}v,A^{2k}v>-<v,v>}{2}=0.$$
\begin{align*}
<A^{2k+1}v,v>=&\omega^{-1}<\omega A^{2k+1}v, v>=\frac{\omega^{-1}}{2}(<(\omega A^{2k+1}+I)v,(\omega A^{2k+1}+I)v>-\\
&-<\omega A^{2k+1}v,\omega A^{2k+1}v>-<v,v>)=0.
\end{align*}
Note that indeed it follows from $A^* = -A$ that $<A^{2k}v, v> = <v, A^{2k}v>$ and that $<\omega A^{2k+1}v,v> = <v,\omega A^{2k+1}v>$.

Now for the case of $\mathfrak{o}$, we still have  $g_1$, and the same proof as before shows that $<A^{2k}v,v>=0$.
But it is always true that 
$$<A^{2k+1}v,v>=<v, -A^{2k+1}v> = -<A^{2k+1}v,v>,$$
and hence $<A^{2k+1}v,v>=0$.

For the case of $\mathfrak{sp}$, we use the same technique but to the condition imposed from Lemma \ref{Av perp v}. This way we get that for a Zariski dense subset of $\FF[x]$ (and thus for all $g\in \FF[x]$) that $<Ag(A^2)v, g(A^2)v> = 0$. From this we are able to get:

\begin{align*}
<A^{2k+1}v,v>=&\frac{1}{2}(<A(A^{2k}+I)v,(A^{2k}+I)v>-\\
&-<A\cdot A^{2k}v,A^{2k}v>-<Av,v>)=0.
\end{align*}
Moreover, it is always true that
$$<A^{2k}v,v>=<v, A^{2k}v> = -<A^{2k}v,v>,$$
and hence $<A^{2k}v,v> = 0$.
\end{proof}

\begin{lemma}
Given $(A,v)\in R$, we have \begin{enumerate}[label=(\roman*)]
\item For any $\lambda\in \KK$ for which $\nu_\lambda$ is defined, $\Delta(\nu_\lambda(A,v))=\Delta(A,v)$.
\item For any $\lambda\in \FF$, $\Delta(\mu_\lambda(A,v))=\Delta(A,v)$.
\end{enumerate}
\end{lemma}
\begin{proof}
For $\Delta(\nu_\lambda(A,v))$, this follows directly from Proposition \ref{matrix determinant}.
For $\Delta(\mu_\lambda(A,v))$, this also follows from Proposition \ref{matrix determinant}, but with an iterative use:\\
Since $<A^k Av, v>=0$ for all $k\geq 0$, $\Delta(A)=\Delta(A+A\phi_v)$. To prove that we also have $\Delta(A+A\phi_v)=\Delta(A+A\phi_v+\phi_v A)$, we must check that $<(A+A\phi_v)^k v, A^* v>=0$ for all $k\geq 0$. Now for any $k\geq 0$ we have $\phi_v A^k v = <A^k v, v>v = 0$, so
$$<(A+A\phi_v)^k v, A^* v>=<A^k v, A^* v> = <A^{k+1} v, v> = 0.$$
\end{proof}

\section{Stratification}\label{stratification}
For any $g\in \KK [x]$ which is a power of an irreducible polynomial, let $Y_g$ be the subspace of $\mathfrak{g}$ consisting of elements with characteristic polynomial $g$.
By the Localization principle (Corollary \ref{LocPrinCor}), the previous reformulations, and Theorem \ref{HC dec 2}, it is enough to prove that any $(\tG, \chi)$-equivariant distribution on $\Delta^{-1}(g)=Y_g\times V$ is $0$, for any $g$ as above. Let us fix $g$ and prove this claim for it.

We proceed similarly to \cite{AGRS} and \cite{GL case}. The strategy will be to stratify $Y_g$ and restrict stratum by stratum the possible support for a  $(\tG, \chi)$-equivariant distribution (note that $Y_g$ is a union of finitely many $\tG$ orbits). For the unitary case, choose $\omega\in \KK$ s.t. $\bar{\omega}=-\omega $ (in the symplectic case denote $\omega = 1$). For $\lambda\in \FF$, denote by $\eta_\lambda$ either $\nu_{\lambda\omega}$ or $\mu_\lambda$, depending on which case we are in.
\begin{notation}
Denote by $P_i(g)$ the union of all $\tG$-orbits of $Y_g$ of dimension at most $i$, and let $R_i(g):=R\cap (P_i(g) \times V)$. Also, for any open $\tG$-orbit $O$ of $P_i(g)$ set $$\tO:=(O\times V) \cap \bigcap_{\lambda\in\FF}\eta_{\lambda}^{-1}(R_i(g)).$$
Note that $P_i(g)$ are Zariski closed inside $Y_g$, $P_k(g)=Y_g$ for $k$ big enough, and $P_{-1}(g)=\emptyset$.
\end{notation}

We denote by $\Fou_V$ the Fourier transform on $V$ with respect to the non-degenerate $\FF$-bilinear form $(u,v)\mapsto \operatorname{tr}_{\KK/\FF}(<u,v>)$. It will also be used to denote the partial Fourier transform on $V$ when applied to $X\times V$ for some space $X$.
In the cases of $\mathfrak{g} = \mathfrak{o}$ and $\mathfrak{g} = \mathfrak{u}$, $\Fou_V$ commutes with the action of $\tG$. In the case of $\mathfrak{g} = \mathfrak{sp}$, it is not true.
Instead, the action on $\tG$ after applying $\Fou_V$ is compatible with the action of $\tG$ on $V$ by $(g, \delta).v:= gv$ (recall that the usual action of $\tG$ on $V$ is by $(g, \delta).v:= \delta gv$).
Since $-1\in Sp$, we still have that Fourier transform maps $\Sc^*(X\times V)^{H, \tau}$ into itself for any $X\ss \mathfrak{sp}$, any subgroup $H$ of $\tSp$ containing $-1$, and any $\tau\in \{1, \chi\}$.
\begin{claim} \label{orbit}
Let $g\in \KK[x]$ be a polynomial which is a power of an irreducible polynomial $f$ satisfying $f = \pm f^*$.
Let $O$ be an open $\tG$-orbit of $P_i(g)$. Suppose $\xi$ is a $(\tG, \chi)$-equivariant distribution on $O\times V$ such that
$$\mathrm{supp}(\xi) \ss \tO ,$$
and
$$\mathrm{supp}(\Fou_{V}(\xi)) \ss \tO .$$
Then $\xi=0$.
\end{claim}
This claim will be proven in the next section.\\
Let us now show how it implies the main theorems. Recall that Theorem \ref{main goal for Lie Alg}, which states that any $(\tG, \chi)$-equivariant distribution on $\mathfrak{g}\times V$ is 0,
implies Theorems \ref{G(V), G(W) dist} and  \ref{J dist}. This is by virtue of Theorem \ref{linearization} and what is shown in Section \ref{reformulations}.
\begin{proof}[Proof of Theorem \ref{main goal for Lie Alg}]
We prove the following claim by downward induction - any $(\tG, \chi)$-equivariant distribution on $\Delta^{-1}(g)$ is supported inside $R_i(g)$. This claim for $i$ big enough follows from Proposition \ref{inside R}, and the claim for $i=-1$ implies Theorem \ref{main goal for Lie Alg} by the localization principle (Corollary \ref{LocPrinCor}) and Proposition \ref{HC dec 2}, as already explained in the top of this section.
For the induction step, take such a distribution $\xi$. As $P_i(g)\setminus P_{i-1}(g)$ is a disjoint union of open orbits, it is enough to show that the restriction of $\xi$ to any $O\times V$, where $O$ is an open orbit of $P_i(g)$, is zero. Let $\zeta=\xi|_{O\times V}$ be such a restriction.
By the induction hypothesis applied to $\eta_{\lambda}(\xi)$, we know that  $\mathrm{supp}(\zeta) \ss \tO$ and similarly $\mathrm{supp}(\Fou_{V}(\zeta)) \ss \tO$. Hence by Claim \ref{orbit}, $\zeta=0$.
\end{proof}

\section{Handling a single stratum - proof of Claim \ref{orbit}}\label{single stratum}
\subsection{Nice operators}
This subsection closely follows \cite[Section 6]{AGRS} and \cite[Subsection 4.3]{GL case} but we give it here for completeness.

\begin{notation}
For $A\in \mathrm{gl}(V)$, set in the cases $\mathfrak{g} = \mathfrak{u}, \mathfrak{sp}$:
$$Q_A:=\{v \in V | \phi_v \in [A,\mathfrak {g} (V)]\}.$$
In the case $\mathfrak{g} = \mathfrak{o}$, set:
$$Q_A:=\{v \in V | A\phi_v + \phi_v A \in [A,\mathfrak {g} (V)]\}.$$
Here $[B,C]:=BC-CB$ is the Lie bracket, and $[A,\mathfrak {g} (V)]=\{[A,B]|B\in \mathfrak {g} (V)\}$.  
\end{notation}

\begin{proposition}\label{inside Q}
If $(A,v)\in\tO$ then $v\in Q_A$.
\end{proposition}

\begin{proof}
Consider a point $(A,v)\in \tO$.
The Zariski tangent space to $O$ at $A$ is $[A,\mathfrak {g} (V)]$. Denote by $A_\lambda$ the operator $A+\lambda\omega \phi_v$ in the unitary case, $A+\lambda\phi_v$ in the symplectic case, or $A + \lambda(A \phi_v + \phi_v A)$ in the orthogonal case. Since $A_\lambda$ is contained in $O$ for $\lambda$ small enough (as $\eta_{\lambda}$ keeps $A$ inside $P_i(g)$, in which $O$ is open), we get that $\phi_v\in [A,\mathfrak {g} (V)]$ (or $A\phi_v + \phi_v A \in [A, \mathfrak{g}(V)]$ in the orthogonal case).
\end{proof}

\begin{theorem} \label{QA-RA}
Unless we are in the case $\mathfrak{g} = \mathfrak{o}$ and the characteristic polynomial of $A$ is equal to $x^n$, we have $Q_A \ss R_A$. In this case, we still have $<A^{k}v, v>=0$ for any $v\in Q_A$ and $k\geq 1$.
\end{theorem}
\begin{proof}
For the unitary and symplectic cases:
Assume that $\phi_v= [A,B]$, for some $B \in G$. Then
$$<A^k v, v> = \tr A^k \phi_v = \tr[A, A^k B] = 0.$$
For the orthogonal case:
Assume that $A\phi_v + \phi_v A = [A,B]$, for some $B \in G$. Then
\begin{equation*}
\begin{split}
\tr A^k A\phi_v & =\tr A^k \phi_v A = \frac{\tr (A^{k+1}\phi_v + A^k \phi_v A)}{2} = \frac{\tr A^{k}(AB - BA)}{2}=\frac{\tr[A, A^{k+1}B]}{2}=\\
&=0.
\end{split}
\end{equation*}
Now $\tr A^k A\phi_v=<A^{k+1}v, v>$, so we know for any $k\geq 1$ that $<A^{k}v, v>=0$. If the characteristic polynomial $g$ (which is a power of an irreducible polynomial) is not a power of $x$, then there is a polynomial $h(A)$ s.t. $Ah(A) = \operatorname{Id}$. This implies that $<v, v>=<Ah(A)v, v> = 0$, and so $v\in R_A$.
\end{proof}

\begin{notation}
Let $A \in \mathfrak{g}$. We denote by
$C_A$ the stabilizer of $A$ in $G$ and by $\tC_A$ the stabilizer of $A$ in $\tG$.
\end{notation}
\begin{lemma} \label{tc}
For any $A\in\mathfrak{g}$, $\tC_A\neq C_A$.
\end{lemma}
\begin{proof}
We give here the proof for all cases except $G=SO$. The proof for $G=SO$ is given in Appendix \ref{SO centralizer}. By \cite[Chapter 4, Proposition 1.2.]{MVW}, There exists an $\FF$-linear map $T:V\to V$ which satisfies $TAT^{-1} = -A$, and such that for any $u,v\in V$, we have that $<Tu, Tv> = <v, u>$ (this condition implies that $s\lambda u = \overline{\lambda u}$).
Consider $s=(T,-1)$ as an element of $\tG$. Then $s.A= -TAT^{-1} = A$. Thus $s\in \tC(A)\setminus C_A$.
\end{proof}
Thus the $\tG$-orbit of $A$ is equal to its $G$-orbit.
It is known that $C_A$ is unimodular and hence $\tC_A$ is
also unimodular.
Claim \ref{orbit} follows now from Frobenius descent (Theorem \ref{Frob}), Proposition \ref{inside Q} and the following proposition.

\begin{proposition}
Let $A \in \mathfrak{g}$. Let $\eta \in \Sc^*(V)^{C_A}$. Suppose that both $\eta$ and
$\Fou_{V}(\eta)$ are supported in $Q_A$. Then $\eta \in \Sc^*(V)^{\tC_A}$.
\end{proposition}

\begin{definition}
Call an element $A\in \mathfrak{g}$ 'nice' if the previous proposition holds for $A$. Namely, $A$ is 'nice' if any distribution $\eta \in \Sc^*(V)^{C_A}$ such that both $\eta$ and $\Fou(\eta)$ are supported in $Q_A$ is also $\tC_A$-invariant.
\end{definition}

\begin{lemma} \label{DirectSum}
Let $A_1 \in \mathfrak{g}(V_1)$ and $A_2 \in \mathfrak{g}(V_2)$ be nice. Then $A_1 \oplus A_2 \in \mathfrak{g}(V_1\oplus V_2)$ is nice.
\end{lemma}
\begin{proof}
See the proof of \cite[Lemma 6.3]{AGRS}.
\end{proof}
\subsection{A 'simple' operator is nice}

Using the classification of Proposition \ref{full classification}, we need to check that simple non-split, simple even nilpotent, and simple odd nilpotent blocks are nice (recall that we assumed the characteristic polynomial of our original operator to be a power of an irreducible polynomial, thus we need not check simple split operators).
Let $A$ be a block of one of these types.
Let $s = (T, -1)$ be an element of $\tC_A$ with $\chi(s)=-1$. We have $A = s.A = -TAT^{-1}$, and so $TA = -AT$
We need to prove the following claim for each of the possible block types.

\begin{claim} \label{simple blocks are nice}
Let $\xi$ be a $C_A$-invariant distribution on $V$, such that both $\xi$ and $\Fou(\xi)$ are supported on $Q_A$. Then $\xi$ is also $s$-invariant.
\end{claim}
We shall prove this claim in the following subsections.
This claim implies Claim \ref{orbit}.

\subsubsection{Simple non-split blocks}
Assume that $A$ is a simple non-split block with minimal polynomial $f^d$, $f$ irreducible, and $f^*=\pm f$. If we are in the case $\mathfrak{g} = \mathfrak{o}$, assume also that $f(x)\neq x$.
We know by Proposition \ref{QA-RA} that $Q_A\ss R_A$.
Consider the self-dual increasing filtration $V^i=\ker f(A)^i$. One can easily see that $R_A=V^{\floor{d/2}}$. The fact that $\Fou(\xi)$ is supproted on $V^{\floor{d/2}}$ means that $\xi$ is invariant to shifts by $(V^{\floor{d/2}})^{\perp}=V^{\ceil{d/2}}$. Now consider two cases:
\begin{enumerate}
    \item $d$ is odd. Then $V^{\floor{d/2}}\subsetneq V^{\ceil{d/2}}$. Choosing a vector $v\in V^{\ceil{d/2}}\setminus V^{\floor{d/2}}$, we get that $\xi$ is the same as $\xi$ shifted by $v$, and that it is supported on $V^{\floor{d/2}}\cap (v+V^{\floor{d/2}})=\emptyset$, thus $\xi=0$.
    \item $d$ is even. Then $V^{\floor{d/2}}= V^{\ceil{d/2}}=V^{d/2}$. Thus $\xi$ is the extension by 0 of $\xi |_{V^{d/2}}$, which is a shift invariant distribution on $V^{d/2}$. Thus it is a multiple of the Lebesgue measure on $V^{d/2}$. So it is left to check that if the Lebesgue measure $\zeta$ on $V^{d/2}$ is $C_A$ invariant it is also $s$-invariant. For this, first check that $V^{d/2}$ is $T$-invariant (and so $s$ invariant):
    \begin{equation*}
    \begin{split}
        TV^{d/2} &=Tf(A)^{d/2}V= \bar{f}(TAT^{-1})^{d/2}TV = \bar{f}(-A)^{d/2}V = f(A)^{d/2}V=V^{d/2}.
    \end{split}
    \end{equation*}
    
    So $s$ multiplies $\zeta$ by a constant $c$, which is positive, because $s$ preserves the positivity of the Lebesgue measure. Since $s^2\in C_A$, we have by assumption $s^2\zeta=\zeta$. So unless $\zeta=0$, $c^2=1$, hence by positivity $c=1$, and we're done.
\end{enumerate}

\subsubsection{Simple non-split nilpotent blocks in the orthogonal case}
Note that a simple non-split nilpotent block is such that $V$ has a basis of the form $e, Ae, \dots, A^{d-1}e$, the minimal polynomial of $A$ is equal to $x^d$, and for some non-zero constant $c\in \FF$,
$$<A^i e, A^j e> =
\begin{cases}
(-1)^j c & i + j = d-1\\
0 & \mathrm{otherwise}
\end{cases}.
$$
Note that this implies that $d$ must be odd.
Let $A$ be such a block. Denote
\begin{align*}
    V_1:= & \Span(e, \dots, A^{(d-3)/{2}}e),\\
    V_2:= & \Span(A^{(d-1)/{2}}e),\\
    V_3:= & A^{(d+1)/{2}}V = \Span(A^{(d+1)/{2}}e, \dots, A^{d-1}e).
\end{align*}
Then by Theorem \ref{QA-RA}, $Q_A\ss V_2\oplus V_3$. So $\xi$ is supported on $V_2\oplus V_3$, and is invariant to shifts by $(V_2\oplus V_3)^\perp = V_3$. So it is of the form $\delta_1 \otimes  R \otimes dv_3$, where $dv_3$ is the Lebesgue measure on $V_3$, $\delta_1$ is the Dirac measure at $0$ on $V_1$, and $R$ is some distribution on $V_2$.
Since it is enough to prove our claim for any valid choice of $s = (T, -1)\in \tC_A$, we may simply take $T(A^i e) = (-1)^{(d+1) / 2 + i} A^i e$. Then $s$ acts on $V$ by $A^i e\mapsto  (-1)^{(d-1) / 2 + i} A^i e$. The spaces $V_1, V_2, V_3$ are $s$ invariant, and $s$ acts on $V_2$ by identity. It is also clear that $dv_3, \delta_1$ are $s$-invariant. Thus $\xi$ is $s$ invariant, and we are done.

\subsubsection{Simple even nilpotent blocks}
Let $A\in \mathfrak{o}$ be a simple even nilpotent block.
Denote
$$E:=\Span(e, Ae, \dots, A^{d-1}e)$$
and
$$F:=\Span(f, Af, \dots, A^{d-1}f).$$
Denote also
$$E_1:=\Span(e, Ae, \dots, A^{d/2 - 1}e), E_2:=\Span(A^{d/2}e, A^{d/2 + 1}e, \dots, A^{d-1}e)$$
and
$$F_1:=\Span(f, Af, \dots, A^{d/2 - 1}f), F_2:=\Span(A^{d/2}f, A^{d/2 + 1}f, \dots, A^{d-1}f).$$
Let $P:V\to V$ be defined by $P A^i e = A^i f, P A^i f = A^i e$. So $PA = AP$.
For any two vectors $u,w\in V$, define the linear operator $\phi_{u,w}v := <v, w>u$.
For any $X\in \operatorname{End}(V)$, denote by $X|_E$ the linear operator from $E$ to itself which sends $v\in E$ to $u_E$, where $Xv = u_E + u_F$ and we have $u_E\in E, u_F\in F$.
Let $v\in Q_A$. By definition we have $A\phi_{v, v} + \phi_{v, v}A = [A, B]$ for some $B\in \mathfrak{g}$.
We have
\begin{align*}
(P(A\phi_{v, v} + \phi_{v, v}A))|_E =& (P[A, B])|_E = (PAB - PBA)|_E = (APB - PBA)|_E =\\
=& [A|_E, (PB)|_E],
\end{align*}
The last equation following from the fact that $A$ preserves $E$ and $F$.
From this it follows that for any $k\geq 0$,
\begin{align*}
\tr((P(A\phi_{v, v} + \phi_{v, v}A)A^k)|_E) = &\tr([A|_E, (PB)|_E] (A|_E)^k) = \\
 =& \tr([A|_E,  (PB)|_E(A|_E)^k]) = 0.
\end{align*}

However 
\begin{align*}
\tr((PA\phi_{v, v}A^k)|_E) = &\tr((AP\phi_{v, v}A^k)|_E)= \tr(A|_E(P\phi_{v,v}A^k)|_E) =\\=& \tr((P\phi_{v,v}A^k)|_E A|_E) = \tr((P\phi_{v, v}A^{k+1})|_E).
\end{align*}
So this expression is equal to half of the left hand side of the previous equation, and so
$$\tr((P\phi_{v, v}A^{k+1})|_E)=0.$$
But
$$\tr((P\phi_{v, v}A^{k+1})|_E) = \tr(\phi_{Pv, (-A)^{k+1}v}|_E) =\tr(\phi_{Pv_F, (-A)^{k+1}v_F}) <Pv_F, (-A)^{k+1}v_F>,$$
where $v = v_E + v_F$, $v_E\in E, v_F\in F$.
Thus for any $k\geq 1$ we have $<A^k Pv_F, v_F> = 0$. Similarly, $<A^k Pv_E, v_E> = 0$.
This implies $v\in E_2\oplus F_2$.
Thus $\xi$ is a distribution as in the statement of Claim \ref{simple blocks are nice}, it must be supported on $E_2\oplus F_2$ and invariant to translations by $(E_2\oplus F_2)^\perp = E_2\oplus F_2$. Thus it is equal to a multiple of the Lebesgue measure on $E_2\oplus F_2$.
As it is enough to prove the claim for any specific choice of $s=(T, -1)\in \tC_A$, we can choose $T$ to be $$A^i e\mapsto (-1)^i A^i e, A^i f \mapsto (-1)^{i+1} A^i f.$$
For this choice, $s$ acts by $A^i e\mapsto (-1)^{i+1} A^i e, A^i f \mapsto (-1)^i A^i f$, and so fixes $\xi$ as desired.

\subsubsection{Simple odd nilpotent blocks}
Let $A\in \mathfrak{sp}$ be a simple odd nilpotent block.
Denote
$$E:=\Span(e, Ae, \dots, A^{d-1}e)$$
and
$$F:=\Span(f, Af, \dots, A^{d-1}f).$$
Denote also
\begin{align*}
E_1:=&\Span(e, Ae, \dots, A^{\frac{d-3}{2}}e),\\
E_2:=&\Span(A^{\frac{d-1}{2}}e),\\
E_3:=&\Span(A^{\frac{d+1}{2}}e, \dots, A^{d-1}e),
\end{align*}
and
\begin{align*}
F_1:=&\Span(f, Af, \dots, A^{\frac{d-3}{2}}f),\\
F_2:=&\Span(A^{\frac{d-1}{2}}f),\\
F_3:=&\Span(A^{\frac{d+1}{2}}f, \dots, A^{d-1}f).
\end{align*}
The following is done similarly to the previous case, of simple even nilpotent blocks.
Let $P:V\to V$ be defined by $P A^i e = A^i f, P A^i f = A^i e$. So $PA = AP$.
For any two vectors $u,w\in V$, define the linear operator $\phi_{u,w}v := <v, w>u$.
For any $X\in \operatorname{End}(V)$, denote by $X|_E$ the linear operator from $E$ to itself which sends $v\in E$ to $u_E$, where $Xv = u_E + u_F$ and we have $u_E\in E, u_F\in F$.
Let $v\in Q_A$. By definition we have $\phi_{v, v}= [A, B]$ for some $B\in \mathfrak{g}$.
We have
\begin{align*}
(P\phi_{v, v})|_E =& (P[A, B])|_E = (PAB - PBA)|_E = (APB - PBA)|_E =\\
=& [A|_E, (PB)|_E],
\end{align*}
From this it follows that for any $k\geq 0$
\begin{align*}
\tr((P\phi_{v, v}A^k)|_E) = &\tr([A|_E, (PB)|_E] (A|_E)^k) = \\
 =& \tr([A|_E,  (PB)|_E(A|_E)^k]) = 0.
\end{align*}

But
$$\tr((P\phi_{v, v}A^k)|_E) = \tr(\phi_{Pv, (-A)^k v}|_E) = <Pv_F, (-A)^k v_F>,$$
where $v = v_E + v_F$, $v_E\in E, v_F\in F$.
Thus for any $k\geq 0$ we have $<A^k Pv_F, v_F> = 0$. Similarly, $<A^k Pv_E, v_E> = 0$.
This implies $v\in E_3\oplus F_3$.
Thus if $\xi$ is an equivariant distribution as in the statement of Claim \ref{simple blocks are nice}, it must be supported on $E_3\oplus F_3$ and invariant to translations by $(E_3\oplus F_3)^\perp = E_2\oplus E_3\oplus F_2\oplus F_3$. This clearly implies that $\xi = 0$.
\appendix
\section{Conjugacy classes and 'simple' elements in the orthogonal and unitary groups} \label{classifications}
Our goal in this appendix is to prove Proposition \ref{full classification}.
We will focus on the classification of $G$-conjugacy classes inside $\mathfrak{g}$, though most of the work, if not all of it, applies also for conjugacy classes in $G(V)$.
The classification will be in three parts. The first is using Lemma \ref{eigenvalues perp} separate the different non-related eigenvalues, the second is to separate the different sizes of rational (or Jordan) blocks, and the third is to separate blocks of the same size one from the other.
\subsection{Separating non-related eigenvalues}\label{separating non-related igenvalues}
Recall Definition \ref{f^*}.


Lemma \ref{eigenvalues perp} has the following corollary, using also Theorem \ref{Rational}:
\begin{corollary}\label{first classification}
Any $A\in\mathfrak{g}$ splits to a direct orthogonal sum of block of two types:
\begin{enumerate}[(A)]
    \item simple split blocks.
    \item Blocks of which the characteristic polynomial is $f^d$, and $f$ is irreducible and satisfies $f^*=\pm f$.
\end{enumerate}
\end{corollary}
Now we continue with the classification of blocks of Type (B).
\subsection{Classifying blocks of Type (B)}
Let $A\in\mathfrak{g}(V)$ be a block of Type (B), with minimal polynomial $f^d$, $f$ irreducible and satisfying $f=\pm f^*$.
\begin{definition}
An operator $X\in\mathfrak{g}(W)$ (for some $W$) which is of Type (X) will be called homogeneous if for any $0\leq j\leq d$ we have $f(X)^j V=\ker f(X)^{d-j}$. That is equivalent to saying that the rational cannonical form of $X$ consists of blocks all of the same size.
\end{definition}
We dedicate this subsection to proving the following Proposition:
\begin{proposition}\label{second classification}
$V$ can be decomposed as an orthogonal sum $V=\bigoplus_{i=1}^d U_i$, and accordingly $A=\bigoplus_{i=1}^d A_i$, such that each $A_i$ is homogeneous.
\end{proposition}
Consider the decreasing filtration $f(A)^i V$ of $V$, and the increasing filtration $V_i=\ker f(A)^i$. We have $f(A)^i V\subseteq V_{d-i}$, and $V_i^{\perp}=f(A)^i V$.
Let $m$ be the minimal integer so that $V_m \subsetneq f(A)V$. For any $0\leq i\leq m-1$, we have $f(A)^i V_m=V_{m-i}$.

\begin{lemma}
For any $0\leq i\leq d$, $V_i^\perp = f(A)^i V$.
\end{lemma}
\begin{proof}
Obviously $f(A)^i V \ss V_i^\perp$. However,
$$\dim f(A)^i V = \dim V - \dim\Ker(f(A)^i) = \dim V - \dim V_i = \dim V_i^\perp.$$
Thus we have $V_i^\perp = f(A)^i V$.
\end{proof}
\begin{lemma}
The form $B$ on $V$ induces a non-degenerate pairing of $f(A)^i V_m / f(A)^{i+1} V_{m+1}$ with $f(A)^{m-i-1}V_m / f(A)^{m-i} V_{m+1}$.
\end{lemma}
\begin{proof}
Let $v\in f(A)^i V_m=V_{m-i}$ and assume that $v\perp f(A)^{m-i-1}V_m=V_{i+1}$. Then $v\in V_{i+1}^{\perp}\cap V_{m-i}=f(A)^{i+1}V\cap V_{m-i}=f(A)^{i+1}V_{m+1}$. The other direction is symmetric.
\end{proof}
\begin{proof}[Proof of Proposition \ref{second classification}]
One can naturally give $V_m / f(A)V_{m+1}$ the structure of a vector space over the field $L:=\KK[A] / f(A)$.
Choose $e_1,\dots, e_k \in V_m$ which are the lifts of a basis of $V_m / f(A)V_{m+1}$ over $L$, and let $U_m:=\Span(A^i e_j)\ss V_m$.
We claim that for any relation of the form $h_1(A)e_1 + \dots + h_k(A)e_k + f(A)^r v = 0$, with $v\in V$, all the polynomials $h_i$ are divisible by $f(A)^{\min(m,r)}$.
To see this, assume otherwise, and rewrite this relation as $f(A)^\ell(f(A)^{r-\ell}v + \sum_{i=1}^k \widetilde h_i(A)e_i) = 0$, where at least one of the polynomials $\widetilde{h}_i(A)$ is not divisible by $f(A)$. Since $\ell < m$, we have $\sum_{i=1}^k \widetilde h_i(A)e_i\in V_\ell + f(A)V = f(A)V$, thus $\sum_{i=1}^k \widetilde h_i(A)e_i$ is a nontrivial relation in $V_m / f(A)V_{m+1}$ over $L$, which is a contradiction.
From this claim it follows that the map $f(A)^i U_m/f(A)^{i+1} U_m \to f(A)^i V_m / f(A)^{i+1} V_{m+1}$ is an isomorphism.
Thus the form $B$ on $V$ induces a non-degenerate pairing of $f(A)^i U_m / f(A)^{i+1} U_{m+1}$ with $f(A)^{m-i-1} U_m / f(A)^{m-i} U_{m+1}$.
In particular, its restriction to $U_m$ is non-degenerate. So $U_m$ splits as an $A$-invariant orthogonal direct summand of $V$.
The restriction of $A$ to $U_m$ is a homogeneous block because
$$f(A)^j U_m=f(A)^j V_m \cap U_m=V_{m-j}\cap U_m=\ker f(A|_{U_m})^{m-j}.$$
Because of the choice of $m$, $\dim U_m>0$.
Also, on $V':=U_m^{\perp}$, the minimal $m'$ s.t. $V'_{m'}\subsetneq f(A)V'$ is bigger than $m$. So by induction we are done.
\end{proof}

\subsection{Decomposing homogeneous blocks to simple non-split blocks}
\begin{lemma}\label{non-isotropic}
Given a non-zero symmetric, Hermitian, or skew-Hermitian form on a vector space $V$, there is a non-isotropic vector.
\end{lemma}
\begin{proof}
Assume that $<v, v>=0$ holds for any $V$. Then
$$0=<u+v, u+v>-<u,u>-<v,v>=<u,v>+<v,u>.$$
Thus the form is skew-symmetric, which contradicts our assumptions.
\end{proof}
\begin{corollary}
Any homogeneous operator decomposes as the direct orthogonal sum of simple non-split blocks, unless it is in the $\mathfrak{o}$ case and its minimal polynomial is $x^d$ for some even $d$, or it is in the  $\mathfrak{sp}$ case and its minimal polynomial is $x^d$ for some odd $d$.
In these cases, it decomposes as the direct sum of simple even nilpotent blocks (respectively simple odd  nilpotent blocks).
\end{corollary}
\begin{proof}
Assume we are neither in the symplectic case nor in the even nilpotent othogonal case which was excluded. Define a non-degenerate sesquilinear form on $U:=V/f(A)V$ by $<u,v>_U:=<f(A)^{d-1}u,v>_V$. If we are in the unitary case, then it is either Hermitian or skew-Hermitian, depending on whether $(f^{d-1})^*=f^{d-1}$ ($f^{d-1}$ is of even degree), or $(f^{d-1})^*=-f^{d-1}$ ($f^{d-1}$ is of odd degree). If we are in the orthogonal case, $(f^{d-1})^*=f^{d-1}$, and so this bilinear form is symmetric. So there is a non isotopic vector in $U$, which means that there is a vector $v\in V$ such that $<f(A)^{d-1}v,v>=\lambda\neq 0$.
Let $V_0=\Span(v, Av, A^2v,\dots)$. It is an $A$-invariant space, to which the restriction of the form $B$ is non-degenerate. It is also generated by one vector ($v$). By induction, we are done.

In the symplectic case which is not the case excluded, we have a similar proof. Again define a bilinear form on $U:=V/f(A)V$. If the minimal polynomial is $x^d$ for $d$ even, define as before $<u,v>_U:=<A^{d-1}u,v>$, and it will be a symmetric form. Otherwise, $A$ is invertible and we set $<u,v>_U:=<Af(A)^{d-1}u,v>$. Again this bilinear form is symmetric. The rest of the proof follows the same, with noticing that $<Af(A)^{d-1}v,v>\neq 0$ implies that $V_0\cap V_0^\perp = 0$ for $V_0=\Span(v, Av, A^2v,\dots)$.

In the orthogonal case, if the minimal polynomial is $x^d$ for $d$ even (respectively the symplectic case and $d$ odd), the biliear form $<u, v>_U:=<A^{d-1}u, v>$ on $U$ is skew-symmetric. Take $u_1, u_2\in U$ s.t. $<u_1, u_2>_U=1$. Lift them to $v_1, v_2 \in V$, and Take $V_0 = \Span(v_1, Av_1, A^2v_1, \dots) \oplus \Span(v_2, Av_2, A^2v_2, \dots)$ (note that this sum is indeed direct). $V_0$ is $A$-invariant, and $V_0^\perp\cap V_0=0$.
Now we are left to show that $V_0$ is a simple even nilpotent block (respectively simple odd nilpotent block).
The first step is to show we can alter the lifts of $u_1, u_2$ to $v_1, v_2$ (from $V_0 / AV_0$ to $V_0$) such that $<A^j v_1, v_1> = 0$ for all $j$ (We will show it for $v_1$, but it is exactly the same for $v_2$). For odd (respectively even) $j$ this holds automatically. We assume for the following that $v_2$ is any lift of $u_2$ (the important property is that $<A^{d-1}v_1, v_2> = 1$). If $m$ is the minimal integer such that $<A^{d-m}v_1, v_1>\neq 0$, we can add a multiple of $A^{m-1}v_2$ to $v_1$ to fix that, as
\begin{equation*}
    \begin{split}
& <A^{d-m}(v_1 + \lambda A^{m-1}v_2), v_1 + \lambda A^{m-1}v_2> = \\
= & <A^{d-m}v_1, v_1> + 2\lambda<A^{d-1}v_2, v_1> - \lambda ^ 2<A^{d+m-2}v_2, v_2> =  \\
= & <A^{d-m}v_1, v_1> - 2\lambda,
    \end{split}
\end{equation*}
$$<A^{d-i}(v_1 + \lambda A^{m-1}v_2), v_1 + \lambda A^{m-1}v_2> = <A^{d-i}v_1, v_1> = 0$$
for any $i<m$, and
$$<A^{d-1}(V_1+\lambda A^{m-1}v_2), v_2>=<A^{d-1}v_1, v_2> = 1.$$
Notice that $m$ must be even and $m\geq 2$.
So by applying this consequtively, we may change the lift of $v_1$ (and similarly $v_2$) in the desired way (notice that indeed we changed it only by vectors in $AV_0$).
Now all that is left is to again change $v_1, v_2$ so that $<A^kv_1, v_2> = 0$ for any $k < d - 1$. For this simply choose a vector $\tilde{v}_2$ which is orthogonal to $v_1, Av_1, A^2v_1, \dots, A^{d-2}v_1, v_2, Av_2, \dots, A^{d-1}v_2$ and such that $<A^{d-1}v_1, v_2> = 1$. Note that the vector $v_2 - \tilde{v}_2$ is perpendicular to the subspace $V_2:=\Span(v_2, Av_2, \dots, A^{d-1}v_2)$, and so $v_2 - \tilde{v}_2\in V_2^\perp = V_2$. This implies that for all $k\geq 0$, $<A^k \tilde{v}_2, \tilde{v}_2> = 0$. Also $v_2 - \tilde{v}_2$ is perpendicular to $A^{d-1}V_0$, and so $v_2 - \tilde{v}_2\in AV_2$.
Thus we can replace $v_2$ with $\tilde{v}_2$, and all of the needed conditions will be satisfied.
\end{proof}

The above immediately implies Proposition \ref{full classification}

\section{On the centralizers in $\widetilde{\operatorname{SO}}$}\label{SO centralizer}
In this appendix we prove Theorem \ref{tc} for the case of $G = \operatorname{SO}(V)$. We need to show that there exists an element $T\in \operatorname{O}(V)$ such that $TAT^{-1} = -A$ and $\det T = (-1)^{\floor{\frac{n+1}{2}}}$.
Assume the decomposition of \ref{full classification}. It is enough to prove Theorem \ref{tc} for each of the blocks, as then taking the direct sum of the elements $T_i$ found gives an element $T\in \operatorname{O}(V)$ with $TAT^{-1} = -A$. If all of the dimensions of the blocks $n_i$ are even, then $\det(\bigoplus T_i) = \prod \det T_i = (-1)^{\sum n_i / 2} = (-1) ^ {n / 2}$. Otherwise there is an odd block, and by replacing $T_i$ by $-T_i$ if needed we can control the sign of the determinant to be as we wish.
Now we need to check each of the simple block types.
\subsection{Simple split blocks}
We have $V = V'\oplus V'^*$, with the natural symmetric bilinear form coming from the pairing. $A = \begin{bmatrix} A'&0\\0&-A'^* \end{bmatrix}$. By the well known theorem claiming  that any square matrix (over any field and of any dimension) is conjugate to its transpose, there is an isomorphism $B: V'^*\to V'$ such that $BA'^*B^{-1} = A'$. Taking $T = \begin{bmatrix} 0&B\\(B^*)^{-1}&0\end{bmatrix}\in \operatorname{O}(V)$, we get 
\begin{align*}
    TAT^{-1} &= \begin{bmatrix} 0&B\\(B^*)^{-1}&0\end{bmatrix} \begin{bmatrix} A'&0\\0&-A'^* \end{bmatrix}\begin{bmatrix} 0&B^*\\B^{-1}&0\end{bmatrix} = \begin{bmatrix} 0&-BA'^*\\(B^*)^{-1}A'&0\end{bmatrix} \begin{bmatrix} 0&B^*\\B^{-1}&0\end{bmatrix} \\
    &=\begin{bmatrix} -BA'^*B^{-1}&0\\ 0&(B^*)^{-1}A'B^*\end{bmatrix}= \begin{bmatrix} -A'&0\\ 0&A'^*\end{bmatrix} = -A.
\end{align*}
Also $\det T = (-1)^{\dim V'} = (-1)^{\dim V/2}$.

\subsection{Simple non-split blocks}
We have $V = \Span(e, Ae, A^2e, \dots)$ for some $e\in V$. Define $T:V\to V$ by $T(g(A)e) = g(-A)e$. It is well defined since $g$ is well defined modulo $f(A)^d$, and $f(-A)=\pm f(A)$. $T$ is an element of $\operatorname{O}(V)$, as
$$<g(-A)e, h(-A)e> = <e, g(A)h(-A)e> = <h(-A)g(A)e, e> = <g(A)e, h(A)e>.$$
Clearly $TAT^{-1} = -A$. Also $\det T = (-1)^{\floor{n/2}}$, which is what we wanted in the non-nilpotent case (where $n$ is even), and in the nilpotent case (where $n$ is odd) we may replace $T$ by $-T$ if needed, in order to achieve the desired sign of the determinant of $T$.

\subsection{Simple even nilpotent blocks}
We can choose $T(A^ie) = (-1)^iA^ie$ and $T(A^if) = (-1)^{i+1}A^if$. It clearly satisfies $TAT^{-1} = -A$ and $T\in \operatorname{O}(V)$. It only remains to note that $\det T = (-1) ^{\dim V / 2}$, which is exactly what we wanted.


\begin{thebibliography}{999}

\bibitem[AAG]{AAG} A. Aizenbud, N. Avni, D. Gourevitch: {\it Spherical pairs over close local fields}.	Commentarii Mathematici Helvetici, \textbf{87} (2012), No. 4, 929-962.
\bibitem[AG]{AG} A. Aizenbud, D. Gourevitch: {\it A proof of the multiplicity one conjecture for $\mathrm{GL}_n$ in $\mathrm{GL}_{n+1}$}, arXiv:0707.2363v2 [math.RT].

\bibitem[AGRS]{AGRS} A. Aizenbud, D. Gourevitch, S. Rallis, G. Schiffmann, {\it Multiplicity one Theorems}. Annals of Mathematics, \textbf{172} (2010), No. 2, 1407-1434. See also arXiv:0709.4215 [math.RT].

\bibitem[AGS]{AGS} A. Aizenbud, D. Gourevitch, E. Sayag, {\it $(O(V\oplus F),O(V))$ is a Gelfand pair for any quadratic space V over a local field F}. Math. Z. 261, 239–244 (2009). https://doi.org/10.1007/s00209-008-0318-5

\bibitem[BCCISS]{BCCISS}
A. Borel, R.W. Carter, C. Curtis, N. Iwahori, T.A. Springer, R. Steinberg, {\it Seminar on Algebraic Groups and Related Finite Groups}, E. IV 2. Springer-Verlag Berlin Heidelberg, 1970

\bibitem[Ber]{Ber} J. Bernstein:
{\it $P$-invariant Distributions on $\mathrm{GL}(N)$ and the
classification of unitary representations of $\mathrm{GL}(N)$
(non-archimedean case)} Lie group representations, II (College Park,
Md., 1982/1983), 50--102, Lecture Notes in Math., \textbf{1041},
Springer, Berlin (1984).

\bibitem[BZ]{BZ} J. Bernstein, A.V. Zelevinsky: {\it Representations of the group $\mathrm{GL}(n, F)$,
where F is a local non-Archimedean field.} Uspekhi Mat. Nauk
\textbf{10}, No. 3, 5-70 (1976).

\bibitem[CS]{CS} F. Chen, B. Sun, {\it Uniqueness of Rankin–Selberg Periods}, International Mathematics Research Notices, Volume 2015, Issue 14, 2015, Pages 5849–5873.

\bibitem[GGP]{GGP} W. T. Gan, B. H. Gross, D. Prasad: {\it Symplectic local root numbers, central critical $L$-values, and restriction problems in the representation theory of classical groups}. Sur les Conjectures de Gross et Prasad. I, Astérisque, no. \textbf{346} (2012), 109.

\bibitem[GK]{GK} I. M. Gelfand and D. A. Kajdan [Kazhdan]: {\it Representations of the group
GL(n,K) where K is a local field}. Lie Groups and Their Representations
(Budapest, 1971), Halsted, New York, 1975, 95 - 118. MR 0404534

\bibitem[M]{GL case} D. Mezer, {\it Multiplicity one theorem for $(\mathrm{GL}_{n+1},\mathrm{GL}_n)$ over a local field of positive characteristic}. Math. Z. (2020).
doi:10.1007/s00209-020-02561-1

\bibitem[MVW]{MVW} C. MŒGLIN, M.-F. VIGNÉRAS, and J.-L. WALDSPURGER, {\it Correspondances de Howe sur un Corpsp-adique, Lecture Notes in Math.} \textbf{1291}, Springer-Verlag, New York, 1987.

\bibitem[SZ]{SZ} B. Sun. C.-B. Zhu {\it Multiplicity one theorems: the Archimedean case}.  Annals of Mathematics (2) \textbf{175} (2012), No. 1, 23 - 44.

\bibitem[Sun]{Sun} B. Sun {\it Multiplicity one theorems for Fourier-Jacobi models}. American Journal of Mathematics \textbf{134}, no. 6 (2012): 1655-1678. doi:10.1353/ajm.2012.0044.

\bibitem[Wald]{Wald} J. L. Waldspurger, {\it Une variante d’un résultat de Aizenbud, Gourevitch, Rallis et Schiffmann}, Astérisque \textbf{346} (2012), 313–318, Sur les conjectures de Gross et Prasad. I.

\bibitem[Wall]{Wall} G. E. Wall, {\it On the conjugacy classes in the unitary, symplectic and orthogonal groups}.
Journal of the Australian Mathematical Society,3(1), 1-62. (1963)
doi:10.1017/S1446788700027622

\end{thebibliography}
\end{document}